\documentclass[10pt,a4paper]{article}

\usepackage{graphicx} 
\usepackage{caption}
\usepackage{subcaption}
\usepackage{amsmath,latexsym,amsfonts,amssymb,mathrsfs,amsthm}

\usepackage{verbatim}

\usepackage{euscript}
\usepackage{algpseudocode,algorithm} 

\usepackage[usenames]{color}
\definecolor{darkred}{RGB}{139,0,0}
\definecolor{darkgreen}{RGB}{0,100,0}
\definecolor{darkmagenta}{RGB}{139,0,139}
\definecolor{darkpurple}{RGB}{110,0,180}
\definecolor{darkblue}{RGB}{40,0,200}
\definecolor{darkorange}{RGB}{255,140,0}

\newcommand{\fa}{{\mathfrak a}}

\def\R{\mathbb{R}}
\def\N{\mathbb{N}}
\def\IC{\mathbb{C}}

\def\IN{\mathbb{N}}

\def\IR{\mathbb{R}}

\def\IE{\mathbb{E}}

\newcommand{\C}{\mathbb{C}}

\newcommand{\Oc}{\mathscr{O}}

\newcommand{\cB}{\mathcal{B}}

\newcommand{\calD}{\mathcal{D}}
\newcommand{\cE}{\mathcal{E}}

\newcommand{\cL}{\mathcal{L}}
\newcommand{\cN}{\mathcal{N}}

\newcommand{\cR}{\mathcal{R}}
\newcommand{\cT}{\mathcal{T}}
\newcommand{\cG}{\mathcal{G}}
\newcommand{\cO}{\mathcal{O}}

\newcommand{\cX}{\mathcal{X}}
\newcommand{\bcX}{\mathcal{X}}
\newcommand{\bcY}{\mathcal{Y}}
\newcommand{\cY}{\mathcal{Y}}

\newcommand{\bpi}{{\boldsymbol \pi}}

\newcommand{\eps}{\varepsilon}

\newcommand{\bsnull}{{\boldsymbol 0}}

\newcommand{\bsalpha}{{\boldsymbol{\alpha}}}
\newcommand{\bsbeta}{\boldsymbol{\beta}}
\newcommand{\bsgamma}{{\boldsymbol{\gamma}}}
\newcommand{\bsrho}{{\boldsymbol{\rho}}}
\newcommand{\bstau}{\boldsymbol{\tau}}

\newcommand{\bsnu}{{\boldsymbol{\nu}}}

\newcommand{\bsb}{{\boldsymbol{b}}}

\newcommand{\bsk}{{\boldsymbol{k}}}

\newcommand{\bsx}{{\boldsymbol{x}}}
\newcommand{\bsy}{{\boldsymbol{y}}}
\newcommand{\bsz}{{\boldsymbol{z}}}

\newcommand{\rd}{\mathrm{d}}
\newcommand{\dd}{\mathrm{d}}
\newcommand{\bbQ}{\mathbb{Q}}
\newcommand{\bbR}{\mathbb{R}}
\newcommand{\bbC}{\mathbb{C}}
\newcommand{\bbJ}{\mathbb{J}}
\newcommand{\bbZ}{\mathbb{Z}}
\newcommand{\bbN}{\mathbb{N}}
\newcommand{\bbE}{\mathbb{E}}

\newcommand{\calM}{\mathcal{M}}
\newcommand{\calO}{\mathcal{O}}
\newcommand{\calU}{\mathcal{U}}
\newcommand{\calW}{\mathcal{W}}
\newcommand{\setu}{\mathrm{\mathfrak{u}}}
\newcommand{\setv}{\mathrm{\mathfrak{v}}}

\newcommand{\KL}{Karhunen-Lo\`eve }

\newcommand{\setD}{{\mathfrak{D}}}
\newcommand{\mask}[1]{{}}
\newcommand{\bszero}{{\boldsymbol{0}}}

\numberwithin{equation}{section}

\theoremstyle{plain}
\newtheorem{theorem}{Theorem}[section]

\newtheorem{corollary}{Corollary}[section]
\newtheorem{proposition}{Proposition}[section]
\newtheorem{assumption}{Assumption}
\theoremstyle{definition}
\newtheorem{definition}{Definition}[section]
\newtheorem{remark}{Remark}[section]

\newcommand{\bsone}{{\boldsymbol{1}}}

\newcommand{\calS}{\mathcal S}

\newcommand{\be}{\begin{equation}}
\newcommand{\ee}{\end{equation}}
\newcommand{\ba}{\begin{array}}
\newcommand{\ea}{\end{array}}
\newcommand{\beas}{\begin{eqnarray*}}
\newcommand{\eeas}{\end{eqnarray*}}
\newcommand{\bea}{\begin{eqnarray}}
\newcommand{\eea}{\end{eqnarray}}

\title{Higher order Quasi-Monte Carlo integration for 
       \\
       Bayesian Estimation}
\author{
Josef Dick, 
Robert N. Gantner,
Quoc T. Le Gia,
and 
Christoph Schwab
}

\date{\today}

\begin{document}
\maketitle
\begin{abstract}
We analyze combined Quasi-Monte Carlo quadrature and Finite Element approximations
in Bayesian estimation of solutions to countably-parametric operator equations
with holomorphic dependence on the parameters as considered in
[Cl.~Schillings and Ch.~Schwab: 
 Sparsity in Bayesian Inversion of Parametric Operator Equations.
 Inverse Problems, {\bf 30}, (2014)].
Such problems arise in numerical uncertainty quantification
and in Bayesian inversion of operator equations
with distributed uncertain inputs, such as uncertain coefficients, 
uncertain domains or uncertain source terms and boundary data.
We show that the parametric Bayesian posterior densities
belong to a class of weighted Bochner 
spaces of functions of countably many variables,
with a particular structure of the QMC quadrature 
weights:
up to a (problem-dependent, and possibly large) 
finite dimension $S$
product weights can be used, and beyond this dimension,
weighted spaces with so-called SPOD weights, recently
introduced in 
[J.~Dick, F.Y.~Kuo, Q.T.~Le Gia, D.~Nuyens and Ch.~Schwab, Christoph Higher order QMC Petrov-Galerkin discretization for affine parametric operator equations with random field inputs. SIAM J. Numer. Anal. 52 (2014),  2676--2702.], 
are used to describe the solution regularity.
We establish error bounds for higher order Quasi-Monte Carlo quadrature
for the Bayesian estimation based on
[J.~Dick, Q.T.~LeGia and Ch.~Schwab,
 Higher order Quasi-Monte Carlo integration for holomorphic, parametric
 operator equations, Report 2014-23, SAM, ETH Z\"urich].
It implies, in particular, 
regularity of the parametric solution and of
the countably-parametric Bayesian posterior density 
in SPOD weighted spaces.
This, in turn, implies that the
Quasi-Monte Carlo quadrature methods in 
[J. Dick, F.Y.~Kuo, Q.T.~Le Gia, D.~Nuyens, Ch.~Schwab, 
Higher order QMC Galerkin discretization for 
parametric operator equations, SINUM (2014)]
are applicable to these problem classes,
with dimension-independent convergence rates $\calO(N^{-1/p})$ of 
$N$-point HoQMC approximated Bayesian estimates, where $0<p<1$
depends only on the sparsity class of the uncertain input in the Bayesian estimation.
Fast component-by-component (CBC for short) construction 
[R. N. Gantner and Ch. Schwab Computational Higher Order Quasi-Monte Carlo Integration, 
 Report 2014-25, SAM, ETH Z\"urich]
allow efficient deterministic Bayesian estimation 
with up to $10^4$ parameters.
\end{abstract}

\newpage
\setcounter{page}{0}
\tableofcontents
\newpage
\section{Introduction}
\label{sec:Intro}
The statistical estimation of solutions of operator equations 
which depend on \emph{uncertain inputs}, subject to given noisy data,
is a key task in computational uncertainty quantification. 
In the present paper we consider the particular case when 
the uncertain input is {\em distributed}.
Specifically, we allow
the distributed, uncertain input $u$ to take values in an
infinite-dimensional, separable Banach space $X$.
The {\em forward responses} resulting from 
(instances of) uncertain data $u\in X$ then
take values in a second Banach space, 
the \emph{state space} $\bcX$.
In Bayesian estimation, one is interested in computing
the expected value of a Quantity of Interest (QoI for short)
taking values in $\IR$.
The mathematical expectation (or \emph{ensemble average}) is
conditional on the given, noisy observation data $\delta \in Y$.

The efficient computation of such QoI's in either
forward or inverse problems involves two basic steps: 
i) approximate (numerical) solution of the operator equation 
   in the forward problem, and 
ii) approximate evaluation of the mathematical expectation 
    w.r.t. the posterior over all possible realizations of the
    uncertain input, conditional on given data, 
    by some form of numerical integration. 

Due to the high (infinite) dimensionality of the integration domain,
Monte-Carlo methods have been widely used.
In the present paper, building on our previous work \cite{DKGNS13} on
high-dimensional Quasi-Monte Carlo integration and on 
numerical Bayesian estimation \cite{SS12, SS13} 
we propose a novel deterministic
computational approach towards these aims. 
It consists in
i) \emph{uncertainty parametrization}: through an unconditional
   basis $\{ \psi_j \}_{j\geq 1}$ of $X$, the forward problem
   is transformed formally to an infinite-dimensional, 
   parametric deterministic problem.
ii) \emph{dimension-truncation}: the uncertain input $u$ is restricted
    to a finite number $s$ of parameters.
iii) \emph{(Petrov-)Galerkin discretization} 
    of the parametric operator equation and, finally,
iv) \emph{Quasi-Monte Carlo} (QMC) integration in $s$ parameters 
    from step ii) to compute approximate Bayesian estimates for the
    quantity of interest (QoI).

The present paper is motivated in part by \cite{KSS12}, where QMC
integration using a family of \emph{randomly shifted lattice rules} 
was combined with Petrov-Galerkin 
Finite Element discretization for a model parametric
diffusion equation, and in part by \cite{ScMCQMC12}, 
where the methodology of \cite{KSS12} 
was extended to forward problems described by an abstract family of 
{\em affine-parametric, linear} operator equations. 

The treatment of inverse problems 
is based on the infinite-dimensional Bayesian framework 
as developed in \cite{Stuart10,Stuart13}.
In this work, in contrast to \cite{KSS12,ScMCQMC12}, we 
analyze \emph{deterministic, interlaced polynomial lattice rules} 
for the numerical evaluation of Bayesian estimates.
As we show here, these higher order QMC quadratures 
can provide dimension-independent convergence rates
beyond order one for smooth integrands (cf. \cite{D08,D09}); 
convergence order $1$ was the limitation in \cite{KSS12,ScMCQMC12}
and order $1/2$ 
is an intrinsic limitation of Monte-Carlo based methods 
(here, convergence order is meant in terms of the number $N$ 
of ``samples'', i.e.~of forward solves).
We prove that sparsity of uncertainty parametrization
implies higher order, dimension-independent convergence rates for 
QMC evaluation of ratio estimators for
expectations of QoI's under the Bayesian posterior,
for a broad class of 
smooth, nonlinear, and possibly nonaffine-parametric operator 
equations with distributed uncertain input data.
Our results imply that unlike MCMC and filtering methods,
the presently proposed QMC evaluation of ratio estimators
can provide convergence rates larger than $1/2$ regardless 
of the dimension of the parameter space, while also allowing 
for ``embarrassing parallelism''.

The structure of this paper is as follows:
In Section \ref{sec:HolOpEq}, we introduce a class of 
smooth, nonlinear, holomorphic-parametric operator equations
admitted in our approach.
We present sufficient conditions on the nonlinear operators 
and on the uncertainty for the forward problems to be well-posed, 
as in \cite{GR90,PR}. 
We require these conditions uniformly in the set $\tilde{X}$ 
of admissible uncertainties.
We give a parametrization of the uncertain inputs which reduces
the forward problem with 
distributed uncertain input to a countably-parametric, 
deterministic problem.
These parameters, denoted by $y_j$, are assumed scaled so as 
to take values in $[-1,1]$.

We review several approximations of these equations 
which are required in their computational Bayesian inversion, 
in particular, (Petrov-)Galerkin discretizations
of the parametric forward equations with
discretization error estimates from \cite{GR90,PR}.

In Section \ref{sec:BayInv},
we review Bayesian inversion for these operator equations,
based on \cite{Stuart10,Stuart13} and on \cite{ScSt11,SS12,SS13}.
The countably-parametric representation of the uncertain inputs
allows us to write the integrals arising in Bayesian estimation 
as countably iterated parametric integrals.
The principal result of the present paper, 
a convergence rate bound of the Quasi Monte-Carlo integration
for these integrals, requires precise derivative bounds for the
integrand functions. These are proved based on analytic 
continuation of the integrand functions into the complex domain.
To this end, in Section \ref{sec:Hol}, we review the notion
of {\em holomorphy} of countably parametric integrand 
functions in both forward and inverse problems. 
Section~\ref{sec:anadepsol} gives the holomorphy and 
resulting derivative bounds on parametric forward solutions,
whereas Section \ref{sec:pdTh} contains the corresponding
holomorphy results for the Bayesian posterior densities. 
Section \ref{sec:QMC} reviews recent results on
convergence theory for higher-order QMC quadratures 
(based on \cite{DiPi10}) and for the 
countably-parametric integrands which arise in 
Bayesian estimation (based on \cite{DKGNS13}).
Section \ref{sec:CombErrBd} presents the 
combined error bound
for the QMC-PG approximation of the Bayesian estimate.

Based on the results in Sections \ref{sec:anadepsol} 
and \ref{sec:pdTh}, in \cite{DKGNS13,GaCS14}
variants of the fast 
CBC constructions of generating vectors are developed
based on \cite{DKGNS13,DiGo12,Go13} which are tailored to
the `hybrid' nature of the QMC weights. 

Elliptic model problems illustrating the general theory of 
holomorphic forward maps are presented in Section \ref{sec:modelproblems}.
Numerical results for Bayesian estimation for these problems
are presented in Section \ref{sec:results}.
%
\section{Forward UQ for Parametric Operator Equations}
\label{sec:HolOpEq}

We introduce a class of smooth, nonlinear operator equations
with distributed uncertain input data $u$ taking values
in a separable Banach space $X$.
Upon appropriate uncertainty parametrization, these equations
become countably-parametric operator equations.

\subsection{Operator equations with uncertain input}
\label{sec:OpEqUncInp}
Let $X, \bcX$ and $\bcY$ be real, separable Banach spaces.
For a distributed, uncertain parameter $u\in X$,
assume a {\em nominal parameter instance}
$\langle u \rangle \in X$
(such as, for example, the expectation of an $X$-valued random field $u$),
is known.

Let $B_X(\langle u \rangle;R)$ be an open ball of radius $R>0$ in $X$ 
centered at a nominal input $\langle u \rangle\in X$. 
We consider the following problem: 
\be\label{eq:NonOpEqn}
\mbox{given}\;\;u\in B_X(\langle u \rangle;R),\;
\mbox{find} \; q\in \bcX \quad \mbox{s.t.} \quad 
_{\bcY'}\langle \cR(u;q) , v \rangle_\bcY = 0  \quad \forall v\in \bcY
\;,
\ee
where $\cR: X \times \bcX \rightarrow \bcY'$ is 
the residual of a \emph{forward operator},
depending on $u$ and acting on $q\in \bcX$.

Given $u\in B_X(\langle u \rangle;R)$,
a solution $q_0$ of \eqref{eq:NonOpEqn} is called {\em regular at $u$} 
if and only if the map $\cR(u;\cdot)$ is 
Fr\'{e}chet differentiable with respect to $q$ at $q_0$
and if the differential is an isomorphism:
$(D_q\cR)(u;q_0)\in \cL_{iso}(\bcX;\bcY')$.
Here and in what follows,
$\cL_{iso}$ shall denote the set of linear isomorphisms
between the Banach space arguments.

We assume the map $\cR(u;\cdot):\bcX\to \bcY'$ 
admits a family of regular solutions 
{\em locally, in an open neighborhood of the nominal parameter instance
      $\langle u \rangle \in X$} 
so that the operator equations involving $\cR(u;q)$ are well-posed. 
For all $u$ in a sufficiently small, closed neighborhood
$\tilde{X}\subseteq X$ of  $\langle u \rangle\in X$ 
we impose the following structural assumption
on the parametric forward problem:
\begin{assumption}\label{asmp:LocSolv}
For every $u\in \tilde{X} \subseteq X$,
we assume given
maps $A(u;q): X\times \bcX \to \bcY'$ and $F(u): X \to \bcY'$ 
such that 
\begin{equation}\label{eq:main}
\cR(u;q) = A(u;q) - F(u) \quad \mbox{in}\quad \bcY'
\;.
\end{equation}
We denote by 
$\fa: \bcX \times \bcY \to \IR$ the bilinear form associated with $A$.
For every fixed $u\in \tilde{X}\subset X$,
and for every $F(u)\in \bcY'$, the problem
to find $q(u)\in \bcX$ such that the residual equation
\eqref{eq:NonOpEqn} is well-posed,
i.e.~there exists a unique solution $q(u)$ 
of \eqref{eq:NonOpEqn} which depends continuously on $u$.
\end{assumption}
%
The set 
$\{(u,q(u)): u\in \tilde{X} \} \subset X\times \bcX$ is called a
{\em regular branch of solutions} of \eqref{eq:NonOpEqn} if
\begin{equation}\label{eq:RegBranch}
\begin{array}{l}
\tilde{X}\ni u \mapsto q(u) \;\mbox{is continuous as mapping from}
\; \tilde{X} \to \bcX\;,
\\
\cR(u;q(u)) = 0 \quad \mbox{in}\quad \bcY'
\;.
\end{array}
\end{equation}
%
The regular branch of solutions \eqref{eq:RegBranch} 
is called 
{\em nonsingular} if, in addition, the differential
\begin{equation}\label{eq:NonSingBranch}
(D_q\cR)(u;q(u))\in \cL(\bcX,\bcY') \;
\mbox{is an isomorphism from $\bcX$ onto $\bcY'$,
     for all $u\in \tilde{X}$}
\;.
\end{equation}
%
Well-known sufficient conditions 
for well-posedness of \eqref{eq:NonOpEqn} are stated in the following proposition.
More precisely, for regular branches of nonsingular solutions 
given by \eqref{eq:NonOpEqn} - \eqref{eq:NonSingBranch},
the differential $D_q\cR$ satisfies the so-called 
{\em inf-sup conditions}.
\begin{proposition}\label{prop:WellposInfSup}
Assume that $\bcY$ is reflexive and that,
for some nominal value $\langle u \rangle\in X$ 
of the uncertainty, the operator equation 
\eqref{eq:NonOpEqn} with \eqref{eq:main} 
admits a {\em regular branch of solutions} \eqref{eq:RegBranch}
with $q_0\in \bcX$ denoting the ``nominal'' solution corresponding
to the data $\langle u \rangle \in X$.
Then the differential $D_q\cR$ at $(\langle u \rangle,q_0)$
given by the bilinear map 
$$\bcX \times \bcY \ni (\varphi,\psi) 
 \mapsto \,_{\bcY'}\langle  
 D_q\cR(\langle u\rangle;q_0)\varphi,\psi\rangle_\bcY
$$
is boundedly invertible, 
{\em uniformly with respect to $u\in \tilde{X}$}
where $\tilde{X} \subset X$ is an open neighborhood
of the nominal instance $\langle u \rangle \in X$ 
of the uncertain parameter. 
In particular, there exists a constant $\kappa > 0$
such that there holds
\be\label{eq:DqAinfsup}
\forall u\in \tilde{X}: \quad 
\begin{array}{c} 
    \displaystyle
    \inf_{0\ne\varphi\in \bcX} \sup_{0\ne\psi   \in \bcY}
    \frac{
     _{\bcY'}\langle  (D_q\cR)(u;q_0)\varphi,\psi\rangle_\bcY
    }{
    \| \varphi \|_\bcX \|\psi \|_{\bcY}
    }
    \geq \kappa > 0 \;,
    \\
    \displaystyle
    \inf_{0\ne\psi\in \bcY} \sup_{0\ne\varphi \in \bcX}
    \frac{
     _{\bcY'}\langle  (D_q\cR)(u;q_0)\varphi,\psi\rangle_\bcY
    }{
    \| \varphi \|_\bcX \|\psi \|_{\bcY}
    }
    \geq \kappa > 0 \;,
\end{array}
\ee
and
\be\label{eq:DqAsupsup}
\forall u\in \tilde{X}:
\quad 
\| (D_q\cR)(u;q_0) \|_{\cL(\bcX,\bcY')} 
=
\sup_{0\ne\varphi\in \bcX} 
\sup_{0\ne\psi   \in \bcY}
\frac{
_{\bcY'}\langle (D_q\cR)(u;q_0)\varphi,\psi\rangle_\bcY
}{
\| \varphi \|_\bcX \|\psi \|_{\bcY}
}
\leq \kappa^{-1}
\;.
\ee
\end{proposition}
For every $u\in \tilde{X} \subseteq X$, 
under conditions \eqref{eq:DqAinfsup} and \eqref{eq:DqAsupsup}, 
there exists a unique, regular solution $q(u)$ of \eqref{eq:main}
which is uniformly bounded with respect to $u\in \tilde{X}$
in the sense that there exists
a constant $C(F,\tilde{X}) > 0$ such that 
\begin{equation}\label{eq:LocBdd}
\sup_{u\in \tilde{X}} \| q(u) \|_{\bcX} \leq C(F,\tilde{X})
\;.
\end{equation}
%
The set 
$\{ (u,q(u)): u\in \tilde{X}\}\subset \tilde{X}\times \bcX$
is a {\em regular branch of nonsingular solutions}
when \eqref{eq:DqAinfsup} - \eqref{eq:LocBdd} hold.

If 
the nonlinear functional $\cR$ is Fr\'{e}chet differentiable
with respect to $u$ and Fr\'{e}chet differentiable with respect to $q$
at every point of the regular branch
    $\{ (u,q(u)): u\in \tilde{X}\}\subset \tilde{X}\times \bcX$,
then the mapping relating $u$ to $q(u)$ with the branch
of nonsingular solutions is locally Lipschitz on ${\tilde{X}}$: 
i.e.~there exists a Lipschitz constant $L(F,\tilde{X})$ 
such that
\begin{equation}\label{eq:LocLip}
\forall u,v\in \tilde{X}:
\quad 
\| q(u) - q(v) \|_{\bcX} \leq L(F,\tilde{X}) \| u-v \|_X 
\;.
\end{equation}
This follows from the identity 
$(D_u q)(u) = - (D_q\cR)^{-1} ( D_u\cR)$,
and from the isomorphism property 
$(D_u \cR_q)(\langle u \rangle;q_0) \in \cL_{iso}(\bcX,\bcY')$
which is implied by \eqref{eq:DqAinfsup} and \eqref{eq:DqAsupsup}, 
and from the 
continuity of the differential $D_q\cR$ on the regular branch.
 
In what follows, we will place ourselves in the abstract
setting \eqref{eq:main} with uniformly continuously differentiable
mapping $\cR(u;q)$ in a product of neighborhoods
$B_X(\langle u \rangle;R)\times B_\bcX(q(\langle u \rangle);R)$
of sufficiently small radius $R>0$.
The quantity $q(\langle u \rangle) \in \bcX$ is the corresponding 
regular solution of \eqref{eq:main} at 
the nominal input $\langle u \rangle\in X$.
\subsection{Uncertainty parametrization}
\label{sec:Param}
We shall be concerned with the particular
case where $u\in X$ is a random
variable taking values in a subset $\tilde{X}$ of the Banach space $X$.
We assume that $X$ is separable, infinite-dimensional, 
and admits an unconditional Schauder basis $\{\psi_j\}_{j\geq 1}$:
$X = {\rm span}\{\psi_j: j\geq 1\}$.
Then, every $u\in \tilde{X} \subset X$ can be parametrized 
in this basis, i.e.
\begin{equation}\label{eq:uviapsi}
u = u(\bsy) 
:= \langle u \rangle + \sum_{j\geq 1} y_j \psi_j 
\quad \mbox{for some } \bsy = (y_j)_{j \geq 1} \in U
\;.
\end{equation}
%
%
Examples of representations \eqref{eq:uviapsi} are \KL expansions 
(see, e.g., \cite{ST06,SchwabGittelsonActNum11,Stuart10,Stuart13})
or unconditional Schauder bases (see, e.g., \cite{CiDo72}).
Note that the representation \eqref{eq:uviapsi} is not unique: 
rescaling $y_j$ and $\psi_j$ will not change
$u$. 
We will assume, therefore, throughout what follows
that {\em the sequence $\{ \psi_j \}_{j\geq 1}$ is
such that $U=[-1,1]^\N$}. 
For any $\bsy\in U$, norm-convergence in $X$ 
of the series \eqref{eq:uviapsi} in $X$ 
is implied by the {\em summability condition}
\begin{equation}\label{eq:psumpsi0}
    \sum_{j\geq 1} \| \psi_j \|_X <\infty \;.
\end{equation}
Condition \eqref{eq:psumpsi0} will be assumed throughout 
in what follows.
 
To obtain convergence rate estimates for the discretization
of the forward problem, we shall restrict uncertain inputs
$u$ to sets $X_t\subset X$ of inputs $u$ with 
``higher regularity'' 
(measured in a smoothness scale $\{ X_t \}_{t\geq 0}$ 
with $X=X_0 \supset X_1 \supset X_2 \supset ...$),
so that $u\in X_t$ will imply in Assumption \ref{asmp:LocSolv}
that $F(\cdot)\in \bcY'_t$ and $q(u)\in \bcX_t$, with corresponding 
subspaces $\bcX_t\subset \bcX$ and $\bcY'_t\subset \bcY'$
with extra regularity from suitable scales.

We remark that $u\in X_t$ in general corresponds to stronger decay of 
the $\psi_j$ in \eqref{eq:uviapsi} which is relevant for
optimal convergence estimates for multi-level QMC discretizations.
In the present paper, we consider only single-level algorithms.
For $u\in X$, in \eqref{eq:uviapsi} 
the $\{\psi_j\}_{j\geq 1}$ are thus assumed scaled such that 
\eqref{eq:psumpsi0} is strengthened to
%
\begin{equation}\label{eq:psumpsi}
\bsb := \{ \| \psi_j \|_{X} \}_{j\geq 1} \in \ell^{p}(\N) \;\;\mbox{for some} \;\;0 < p < 1
\;.
\end{equation}
%
We also introduce the 
subset\footnote{For QMC quadrature, ahead, we rescale this set to $[-1/2,1/2]^\bbN$}
\begin{equation}\label{eq:DefSetU}
U 
=
\{ \bsy \in [-1,1]^\bbN: 
u(\bsy) := \langle u \rangle + \sum_{j\geq 1} y_j \psi_j \in \tilde{X} 
\}
\;.
\end{equation}
Once an unconditional Schauder basis $\{\psi_j\}_{j\geq 1}$ has been chosen, 
every realization $u\in X$ can be identified in 
a one-to-one fashion with the pair $(\langle u \rangle,\bsy)$
where $\langle u \rangle$ denotes the 
{\em nominal instance} of the uncertain datum $u$ and 
$\bsy$ is the coordinate vector
of the unique representation \eqref{eq:uviapsi}.
\begin{remark}\label{remk:U=X}
In what follows, by a slight abuse of notation,
{\em 
we identify the subset $U$ in \eqref{eq:DefSetU}
with the countable set of parameters from the 
infinite-dimensional parameter domain 
$U \subseteq \R^\N$
}
without explicitly writing so.
The operator $A(u;q)$ in \eqref{eq:main} then becomes,
via the parametric dependence $u=u(\bsy)$, 
a parametric operator family $A(u(\bsy);q)$ 
which we denote (with slight abuse of notation)
by $\{ A(\bsy;q) : \bsy \in U\}$, with
the parameter set $U=[-1,1]^\bbN$ (again, we use in what
follows this definition 
in place of the set $U$ as defined in \eqref{eq:DefSetU}).
In the particular case that the parametric 
operator family is linear, we have
$A(\bsy;q) = A(\bsy)q$ with $A(\bsy)\in \cL_{iso}(\cX,\cY')$.
We do not assume, however,
that the maps $q\mapsto A(\bsy;q)$
are linear in what follows, unless explicitly stated.
\end{remark}

With this understanding, and under the assumptions
\eqref{eq:LocBdd} and \eqref{eq:LocLip}, 
the operator equation
\eqref{eq:main} will admit, for every $\bsy\in U$, 
a unique solution $q(\bsy;F)$, which is, due to
\eqref{eq:LocBdd} and \eqref{eq:LocLip}, 
uniformly bounded and depends Lipschitz continuously
on the parameter sequence $\bsy\in U$: 
there holds
\begin{equation}\label{eq:ParDepy}
\sup_{\bsy\in U} \| q(\bsy;F) \|_{\cX} \leq C(F,U).
\end{equation}
If the local Lipschitz condition \eqref{eq:LocLip} holds, 
there exists a Lipschitz constant $L>0$ 
such that 
\begin{equation}\label{eq:ParLip}
\| q(\bsy;F) - q(\bsy';F) \|_\cX 
\leq L(F,U) \| u(\bsy) - u(\bsy') \|_X
\;.
\end{equation}
The Lipschitz constant $L>0$ in \eqref{eq:ParLip} 
is not, in general, equal to 
$L(F,\tilde{X})$ in \eqref{eq:LocLip}:
it depends on the nominal instance $\langle u\rangle\in X$
and on the choice of basis $\{\psi_j\}_{j\geq 1}$.
 
\emph{
Unless explicitly stated otherwise, 
throughout what follows, we shall identify 
$q_0 = q(\bsnull;F) \in \cX$ in Proposition \ref{prop:WellposInfSup} 
with the solution of \eqref{eq:NonOpEqn} at the
nominal input $\langle u \rangle \in X$.
}
\subsection{Dimension truncation}
\label{sec:dimtrunc}
For a {\em truncation dimension} $s\in \N$, 
denote the $s$-term truncation of 
parametric representation \eqref{eq:uviapsi} of 
the uncertain datum $u$ by $u^s\in X$.
Dimension truncation is equivalent to setting $y_j=0$ 
for $j>s$ in \eqref{eq:uviapsi} and
we denote by $q^s(\bsy)$ the solution of the corresponding
parametric weak problem \eqref{eq:paraOpEq}. 
Unique solvability of \eqref{eq:paraOpEq} implies 
$q^s(\bsy) = q(\{y_1,y_2,...,y_s,0,...\})$.
For $\bsy\in U$, define $\bsy_{\{1:s\}} := (y_1,y_2,...,y_s,0,0,...)$. 
Proposition~\ref{prop:WellposInfSup}
holds when $u(\bsy)$ is replaced by $u^s(\bsy)$, 
with $\kappa > 0$ in \eqref{eq:DqAinfsup} independent of $s$
for sufficiently large $s$.

Our estimation of the {\em dimension truncation error}
$q(\bsy) - q^s(\bsy)$ relies on two assumptions:
(i) We assume the $p$-summability \eqref{eq:psumpsi} 
of the sequence $\bsb$ given by $b_j := \| \psi_j\|_X$ in \eqref{eq:uviapsi}.
From the definition of the sequence 
$\bsb = (b_j)_{j\ge 1}$ in \eqref{eq:psumpsi},
the condition is equivalent to 
$\sum_{j\ge 1} b_j^p < \infty$;
(ii) the $b_j$ in \eqref{eq:psumpsi} are enumerated so that 
\begin{equation} \label{eq:ordered} 
  b_1 \ge b_2 \ge \cdots \ge b_j \ge \, \cdots\;.
\end{equation}
Consider the $s$-term truncated problem: given $u^s \in \tilde{X}$,
\begin{equation}\label{eq:mainstrunc}
\mbox{find}\;q^s\in \bcX:\quad 
{ _{\bcY'} \langle \cR(u^s;q^s), w \rangle_{\bcY} } = 0 
\;\;\forall w \in \bcY 
\;.
\end{equation}
\begin{proposition} \label{prop:trunc}
Under Assumptions~\eqref{eq:psumpsi0}, \eqref{eq:psumpsi},
for every $f\in \cY'$, for every $\bsy\in U$ and for every $s\in\bbN$, 
the parametric solution $q^s(\bsy)$ of the 
dimensionally truncated, parametric weak problem \eqref{eq:paraOpEq} 
with $s$-term truncated parametric expansion \eqref{eq:uviapsi}
satisfies, with $b_j$ as defined in \eqref{eq:psumpsi},
\begin{equation}\label{eq:Vdimtrunc}
\sup_{\bsy\in U}
  \| q(\bsy) - q^s(\bsy) \|_\bcX
  \,\le\, C(F,X)
  \sum_{j\ge s+1} b_j
\end{equation}
for some constant $C>0$ independent of $f$.
Moreover, for every $G(\cdot)\in \cX'$,
we have
\begin{equation}\label{eq:Idimtrunc}
  |I(G(q))- I(G(q^s))|
  \,\le\, \tilde{C} 
  \sum_{j\ge s+1} b_j,
\end{equation}
where
\begin{equation*}
I(G(q)) = \int_U G(q(\bsy)) \,\mathrm{d} \bsy \mbox{ and } I(G(q^s)) = \int_{[-1,1]^s} G(q(y_1,\ldots, y_s, 0,\ldots)) \,\mathrm{d}y_1 \cdots \,\mathrm{d} y_s,
\end{equation*}
for some constant $\tilde{C}>0$ independent of $s$.
In addition, 
if conditions~\eqref{eq:psumpsi0}, \eqref{eq:psumpsi} 
and \eqref{eq:ordered} hold, then in \eqref{eq:Vdimtrunc}
and \eqref{eq:Idimtrunc} holds
\begin{equation}\label{eq:DTbound}
  \sum_{j\ge s+1} b_j
  \,\le\,
  \min\left(\frac{1}{1/p-1},1\right)
  \bigg(\sum_{j\ge1} b_j^p \bigg)^{1/p}
  s^{-(1/p-1)}
  \;.
\end{equation}
\end{proposition}
\begin{proof}
Assumption \ref{asmp:LocSolv} on well-posedness of the forward problem
\eqref{eq:NonOpEqn} uniformly for all $u\in B_X(\langle u \rangle;R)$
and the basis property \eqref{eq:uviapsi} of the sequence $\{ \psi_j \}$
imply that for sufficiently large $s$, $u^s \in \tilde{X}$
and therefore \eqref{eq:mainstrunc} admits a unique solution, 
$q^s\in \cX$, for these $s$.
The unique local solvability of \eqref{eq:NonOpEqn} and of \eqref{eq:mainstrunc}
implies that $q^s(\bsy) = q(\bsy_{\{ 1: s\}})$ where we recall 
for $\bsy\in U$ the notation $\bsy_{\{ 1: s\}} = (y_1,...,y_s,0,...)$.
From \eqref{eq:ParLip} we obtain 
$$
\begin{array}{l}
\displaystyle
\| q(\bsy;F) - q^s(\bsy;F) \|_\cX 
= 
\|  q(\bsy;F) - q(\bsy_{\{ 1: s\}};F) \|_\cX 
\leq 
L
\sup_{\bsy\in U} \| u(\bsy) - u(\bsy_{\{ 1: s\}}) \|_X
\\
\displaystyle 
=
L
\sup_{\bsy\in U}
\left\| \sum_{j=s+1}^\infty y_j \psi_j \right\|_X
\leq 
L
\sum_{j = s+1}^\infty b_j,
\end{array}
$$
which is \eqref{eq:Vdimtrunc}.
The bound \eqref{eq:Idimtrunc} follows from $G(\cdot) \in \cX'$ and from 
\eqref{eq:Vdimtrunc}.

With $\bsb\in \ell^p(\bbN)$ and the assumption \eqref{eq:ordered},
Stechkin's lemma (see also \cite[p. 3363]{KSS12}) 
implies \eqref{eq:DTbound}.
\end{proof}
%
\subsection{Petrov-Galerkin discretization}
\label{sec:Discr}
Assuming that the infinite-dimensional space $X$ of uncertain inputs 
to admit an unconditional Schauder basis $\{ \psi_j \}_{j\in \bbN}$, 
the uncertain parameter $u\in X$ in 
\eqref{eq:main} can be equivalently expressed as
\eqref{eq:uviapsi} which turns \eqref{eq:main} into an 
{\em equivalent, deterministic, countably parametric operator equation}:
given $\bsy\in U$, find $q(\bsy)\in \cX$ such that
\begin{equation}\label{eq:paraOpEq}
\cR(\bsy;q(\bsy)) = 0 \quad\mbox{in}\quad \cY' \;.
\end{equation}
%
Based on the theory in \cite[Chap. IV.3]{GR90} and
in \cite{PR}, 
an error analysis of Galerkin discretizations of \eqref{eq:paraOpEq}
for the approximation of regular branches of solutions 
of smooth, nonlinear forward problems \eqref{eq:main}
will be presented in this section.
Building on this, in the next section we
generalize the results \cite{KSS12,KSS13} of 
Quasi-Monte Carlo, Petrov-Galerkin approximation 
to direct and inverse problems
for operator equations \eqref{eq:paraOpEq}
with countably-parametric uncertain inputs.

To this end, we assume, as in \cite{ScMCQMC12,DKGNS13}, 
that we are given two sequences
$\{ \cX^h \}_{h>0}\subset\cX$ and $\{ \cY^h \}_{h>0}\subset\cY$ of 
finite dimensional subspaces which are dense in $\cX$ and in $\cY$, respectively. 
For the computational complexity analysis, 
we also assume the following {\em approximation properties}: 
there is a scale $\{ \cX_t \}_{t\geq 0}$ of subspaces such that 
$\cX_{t'} \subset \cX_t \subset \cX_0 = \cX$ for any 
$0<t<t'<\infty$ and such that,
for $0<t \leq \bar{t}$ and $0 < t' \leq \bar{t'}$, 
and for $0< h \leq h_0$, there hold
\begin{equation} \label{eq:apprprop}
\begin{aligned}
\forall v \in \cX_t\;
&:
\quad
\inf_{v^h\in \cX^h} \| v - v^h \|_{\cX}
\,\leq\,
C_t\, h^t\, \| v \|_{\cX_t} \;.
\end{aligned}
\end{equation}
Typical examples of smoothness scales $\{\bcX_t \}_{t\geq 0}$ 
and $\{ \bcY'_t \}_{t\geq 0}$ are
furnished by the Sobolev scale $\bcX_t = H^{1+t}(D)$ 
in smooth domains or by its weighted counterparts in polyhedra \cite{NS12}.
\begin{proposition}\label{prop:stab}
Assume that the subspace sequences 
$\{ \cX^h \}_{h>0}\subset\cX$ and $\{ \cY^h \}_{h>0}\subset\cY$ are stable, 
i.e.~there exist $\bar{\mu} > 0$
and $h_0 > 0$ such that for every $0<h \leq h_0$, there 
hold the uniform (with respect to $\bsy\in U$) 
discrete inf-sup conditions
\begin{align}\label{eq:Bhinfsup1}
&\forall \bsy \in U:
\quad
\inf_{0\ne v^h \in \bcX^h} \sup_{0\ne w^h \in \bcY^h}
\frac{
 _{\bcY'}\langle  (D_q\cR)(u(\bsy);q_0)v^h,w^h\rangle_\bcY
}{
\| v^h \|_\bcX \| w^h \|_{\bcY}
}
\geq \bar{\mu} > 0\;,
\\
\label{eq:Bhinfsup2}
&\forall \bsy\in U:\quad
\inf_{0\ne w^h \in \bcY^h} \sup_{0\ne v^h \in \bcX^h}
\frac{_{\bcY'}\langle  (D_q\cR)(u(\bsy);q_0)v^h,w^h\rangle_\bcY}{\| v^h \|_{\bcX} \|w^h\|_{\bcY}}
\geq \bar{\mu} > 0
\;.
\end{align}
Then, for every $0<h \leq h_0$ 
the Galerkin approximations: given $\bsy\in U$,
\begin{equation} \label{eq:parmOpEqh}
\mbox{find} \; q^h(\bsy) \in \bcX^h :
\quad
{_{\bcY'}}\langle \cR(\bsy;q^h(\bsy)), w^h \rangle_{\bcY} = 0
\quad 
\forall w^h\in \bcY^h\;,
\end{equation}
are uniquely defined and converge quasioptimally:
there exists a constant $C>0$ such that for all $\bsy\in U$
\begin{equation} \label{eq:quasiopt}
 \| q(\bsy) - q^h(\bsy) \|_{\cX}
 \,\le\, 
\frac{C}{\bar{\mu}} \inf_{0\ne v^h\in \cX^h} \| q(\bsy) - v^h\|_{\cX}
\;.
\end{equation}
If $q(\bsy)\in \bcX_t$ uniformly w.r.t.~$\bsy$
and if \eqref{eq:apprprop} holds, then 
\begin{equation}\label{eq:convrate}
\| q(\bsy) - q^h(\bsy) \|_{\cX}
\,\le\, 
\frac{C}{\bar{\mu}} h^t \sup_{\bsy\in U} \| q(\bsy) \|_{\bcX_t} 
\;.
\end{equation}
\end{proposition}
In the ensuing QMC convergence analysis we shall also require
error bounds for the dimensionally truncated parameter sequences.
\begin{corollary}\label{coro:DimTrcStab}
Under the assumptions of  Proposition \ref{prop:stab},
for sufficiently large truncation dimension
$s\in \IN$, for given $\bsy_{\{1:s\}}\in U$ 
the dimensionally truncated Galerkin approximations 
\begin{equation} \label{eq:parmOpEqh_trun}
\mbox{find} \; q^h(\bsy_{\{1:s\}}) \in \bcX^h :
\quad
{{_{\bcY'}}\langle  \cR(\bsy_{\{1:s\}};q^h(\bsy_{\{1:s\}})), w^h \rangle_{\bcY} } = 0
\quad 
\forall w^h\in \bcY^h
\;,
\end{equation}
admit unique solutions $q^h(\bsy_{\{1:s\}}) \in \bcX^h$ which
converge, as $h\downarrow 0$, quasioptimally to $q(\bsy_{\{1:s\}}) \in \bcX$, 
i.e.~\eqref{eq:quasiopt} and \eqref{eq:convrate} hold 
with $\bsy_{\{1:s\}}$ in place of $\bsy$, with the same constants
$C>0$ and $\bar{\mu}$ independent of $s$ and of $h$.
\end{corollary}
%
%
\section{Bayesian Inverse UQ}
\label{sec:BayInv}
The nonlinear, parametric problems 
considered in Section \ref{sec:HolOpEq} were {\em forward problems}: 
for a single instance of the uncertain datum $u\in X$, 
and for given input data $F$,
the quantity of interest was the parametric solution $q(u)$,
or $q(\bsy)$ in terms of the parametrization \eqref{eq:uviapsi}.
Often, however, also the corresponding {\em inverse problem}
is of interest: 
given observational data $\delta$, 
predict a ``most likely'' value of a 
Quantity of Interest (`QoI' for short) 
$\phi$ which, typically, is a continuously 
(Fr\'{e}chet-)differentiable functional of the input $u\in X$.
%
\subsection{General setup}
\label{sec:Setup}
%
Following \cite{Stuart10,Stuart13,ScSt11,SS12,SS13}, 
we equip the space of uncertain inputs $X$ 
and the space of solutions $\cX$ of the forward maps with 
norms $\| \cdot\|_X$ and with $\| \cdot \|_\cX$, 
respectively.
We consider the 
abstract (possibly nonlinear) operator equation \eqref{eq:main}
where the system's forcing $F\in \cY'$ 
is allowed to depend on the uncertain input $u$.

The uncertain operator $A(u;\cdot)\in \cL(\cX,\cY')$ 
is assumed to be boundedly invertible, 
at least locally for the uncertain input $u$ 
sufficiently close to a nominal input 
$\langle u\rangle \in X$, 
i.e.~for $\| u - \langle u \rangle \|_X$ sufficiently small
so that, for such $u$,
the response of the forward problem \eqref{eq:main}
is uniquely defined. 
We define the {\em forward response map},
which maps a given uncertain input $u$ and a given forcing $F$
to the response $q$ in \eqref{eq:main} 
by
$$
X\ni u \mapsto q(u) := G(u;F): X\times \bcY' \to \cX \;.
$$
We omit the dependence of the response on $F$
and simply denote the dependence of the forward solution on 
the uncertain input as $q(u)=G(u)$.
We assume that we are given an observation functional 
$\cO(\cdot): \cX \rightarrow Y$, which denotes a 
{\em bounded linear observation operator} 
on the space $\cX$ of observed system responses in $Y$.
Throughout the remainder of this paper,
we assume that there is a finite number $K$ of
sensors, so that $Y = \IR^K$ with $K<\infty$. 
We equip $Y = \R^K$ with the Euclidean norm, denoted by $|\cdot|$.  
Then~$\cO\in \cL(\bcX; Y) \simeq (\cX^*)^K$, i.e. $\cO(\cdot)$ is a 
$K$-dimensional vector of observation functionals 
$\cO(\cdot) = (o_k(\cdot))_{k=1}^K$.

In this setting, we wish to
predict \emph{computationally} 
an expected (under the Bayesian posterior, defined below)
system response of the QoI, 
conditional on given, noisy measurement data $\delta$.
Specifically, we assume the data $\delta$ 
to consist of 
observations of system responses in 
the data space $Y$, 
corrupted by additive observation 
noise, e.g.~by a realization of a random variable $\eta$ 
taking values in $Y$ with law $\bbQ_0$.
We assume the following form of observed data,
composed of the observed system response
and the additive noise $\eta$
\begin{equation} \label{eq:DatDelta}
    \delta = \cO(G(u)) + \eta \,\in Y \;.
\end{equation}
We assume 
that the additive observation
noise process $\eta$ is Gaussian, 
i.e.~a random vector $\eta \sim \bbQ_0 \sim \cN(0,\Gamma)$
with a positive definite covariance $\Gamma$ on $\R^K$
(i.e., a symmetric, positive definite 
covariance matrix $\Gamma \in \IR^{K\times K}_{sym}$ 
which we assume to be known). 
Henceforth, with a slight abuse of notation, we say $\Gamma > 0$ 
(which means that $\Gamma$ is positive definite).

The {\em uncertainty-to-observation map}
$\cG:X \to \R^K$ of the system is $\cG=\cO \circ G$, so that
\begin{equation*}
    \delta 
    = 
    \cG(u) + \eta = (\cO \circ G)(u) + \eta \,\in L^2_\Gamma(\R^K) \, ,
\end{equation*}
where $L^2_\Gamma(\R^K)$ denotes random vectors taking
values in $Y=\R^K$ which are square integrable with respect
to the Gaussian measure with covariance matrix $\Gamma > 0$ on 
the finite-dimensional observation space $Y=\R^K$.
Bayes' formula \cite{Stuart10,Stuart13} yields a density
of the Bayesian posterior with respect to the prior 
whose negative log-likelihood equals the 
observation noise covariance-weighted, 
least squares functional 
(also referred to as ``potential'' in what follows)
$\Phi_\Gamma:X \times Y \to \bbR$ given by 
$\Phi_\Gamma(u;\delta) = \frac12|\delta - \cG(u) |_{\Gamma}^2$,
i.e.
\begin{equation}\label{eq:DefPhiu}
\Phi_\Gamma(u;\delta)
=
\frac12|\delta - \cG(u) |_{\Gamma}^2
:=
\frac12
\left(
(\delta - \cG(u))^\top \Gamma^{-1} (\delta - \cG(u))
\right)
\;.
\end{equation}
In \cite{Stuart10,Stuart13}, an infinite-dimensional version
of Bayes' rule was shown to hold in the present setting.
In particular, the local Lipschitz assumption \eqref{eq:LocLip}
on the solutions' dependence on the data implies
a corresponding Lipschitz dependence of the Bayesian potential
\eqref{eq:DefPhiu} on $u\in X$. 
Bayes' Theorem states that, under appropriate continuity conditions
on the un\-cer\-tain\-ty-to-ob\-ser\-va\-tion map 
$\cG = (\cO\circ G)(\cdot)$
and on the prior measure $\bpi_0$ on $u\in X$, 
\emph{for positive observation noise covariance}
in \eqref{eq:DefPhiu},
the posterior $\bpi^\delta$ of $u\in X$ given data $\delta\in Y$
is absolutely continuous with respect to the prior $\bpi_0$.
The following result is a version of Bayes' theorem, from 
\cite[Thm. 3.4]{Stuart13}.
\begin{theorem}\label{thm:Bayes}
Assume that the potential 
$\Phi_\Gamma :X\times Y \to \IR$ is, for fixed data $\delta \in Y$,
$\bpi_0$ measurable and that, for $\bbQ_0$-a.e.~data $\delta\in Y$ there holds
$$
Z:= \int_X\! \exp\left( -\Phi_\Gamma(u;\delta) \right) \bpi_0(\dd u) > 0 \;.
$$
Then the conditional distribution of $u|\delta$ ($u$ given $\delta$) exists and
is denoted by $\bpi^\delta$.
It is absolutely continuous with respect to
$\bpi_0$ and there holds
\begin{equation}\label{eq:BayesFormula}
\frac{d\bpi^\delta}{d\bpi_0}(u) 
=
\frac{1}{Z}  \exp\left( -\Phi_\Gamma(u;\delta) \right) 
\;.
\end{equation}
\end{theorem}
In particular, then, the Radon-Nikodym derivative of the 
Bayesian posterior w.r.t.~the prior measure
admits a bounded density w.r.t.~the 
prior $\bpi_0$ which we denote by $\Theta$, and 
which is given by \eqref{eq:BayesFormula}.
%
\subsection{Parametric Bayesian posterior}
\label{eq:ParaPost}
We parametrize the uncertain datum $u$ in the forward equation \eqref{eq:main} 
as in \eqref{eq:uviapsi}. 
Motivated by \cite{SS12,SS13},
the basis for the presently proposed deterministic
quadrature approaches for Bayesian estimation via
the computational realization of Bayes' formula is a 
\emph{
parametric, deterministic representation
}
of the 
derivative of the posterior measure $\bpi^\delta$
with respect to the 
\emph{uniform prior measure $\bpi_0$ on the set $U$ of coordinates
      in the uncertainty parametrization \eqref{eq:DefSetU}}. 
The prior measure $\bpi_0$ being uniform, we admit in \eqref{eq:uviapsi} 
sequences $\bsy$ which take values
in the parameter domain $ U=[-1,1]^\bbN$.
As explained in sections \ref{sec:Setup} and \ref{sec:Param},
we consider the parametric, deterministic 
forward problem in the probability space
\begin{equation}\label{eq:UcBmu0}
(U,\cB,\bpi_0) \;.
\end{equation}
{\em 
We assume throughout what follows that the prior measure 
$\bpi_0$ on the uncertain input $u\in X$, 
parametrized in the form \eqref{eq:uviapsi}, is the uniform measure.
}
With the parameter domain $U$ as in (\ref{eq:UcBmu0}),
the parametric uncertainty-to-observation map 
$\Xi:U \to Y = \bbR^K$ 
is given by
\begin{equation}  \label{eq:DefMapXi}
\Xi(\bsy)
=
\cG(u)\Bigl|_{u= \langle u \rangle+\sum_{j \in \bbJ} y_j\psi_j}
\;.
\end{equation}
Our QMC quadrature approach will be
based on a parametric version of Bayes' Theorem~\ref{thm:Bayes},
in terms of the uncertainty parametrization \eqref{eq:uviapsi}.
To do so, we view $U$ as the unit ball 
of all sequences in $U$ with respect to the $\ell^\infty$-norm, i.e.
the Banach space of bounded sequences taking values in $U$.
\begin{theorem} \label{t:dens}
Assume that $\Xi: U \to Y = \bbR^K$ 
is bounded and continuous, and that $\bpi(\tilde{X}) = 1$.
Then $\bpi^\delta(\dd\bsy)$, the distribution of 
$\bsy\in U$ given data $\delta\in Y$, 
is absolutely continuous with respect to $\bpi_0(\dd\bsy)$,
i.e.~there exists a parametric density $\Theta(\bsy)$ such that
\be \label{eq:post}
\frac{d\bpi^\delta}{d\bpi_0}(\bsy) = \frac{1}{Z} \Theta(\bsy)
\ee
with $\Theta(\bsy)$ given by
\begin{equation} \label{eq:PostDens}
\Theta(\bsy) 
= 
\exp\bigl(-\Phi_\Gamma(u;\delta)\bigr)\Bigl|_{u=\langle u \rangle + \sum_{j \in \bbJ} y_j\psi_j},
\end{equation}
with Bayesian potential $\Phi_\Gamma$ 
as in \eqref{eq:DefPhiu} and 
with normalization constant $Z$ given by
\be \label{eq:Z}
Z 
= \IE^{\bpi_0}\!\left[ 1 \right] 
= \int_{U}\! \Theta(\bsy)\,
\bpi_0(\dd \bsy) > 0 \;.
\ee
\end{theorem}
Bayesian estimation is concerned with the approximation of
a ``most likely'' Quantity of Interest (QoI) $\phi:X \to \IR$
conditional on given (noisy) observation data $\delta\in Y$.
%
With the QoI $\phi$ associate the deterministic, countably-parametric map
\begin{equation}
\Psi(\bsy)
=
\Theta(\bsy)\phi(q(u))\mid_{u= \langle u \rangle + \sum_{j \in \bbJ} y_j\psi_j}
=
\exp\bigl(-\Phi_\Gamma(u;\delta)\bigr)\phi(q(u))
\Bigl|_{ u = \langle u \rangle + \sum_{j \in \bbJ} y_j\psi_j }
: U\rightarrow \IR
\;.
\label{eq:psi}
\end{equation}
Then the Bayesian estimate of the QoI $\phi$, 
given noisy data $\delta$, takes the form
\begin{equation}\label{eq:intpsi}
\IE^{\bpi^\delta}[\phi]
= 
Z'/Z, \;\;
\quad 
Z':=
\int_{U} \! \Psi(\bsy) \,\bpi_0(\dd \bsy)
\;.
\end{equation}
The task in computational Bayesian estimation is 
therefore to approximate the ratio $Z'/Z \in \IR$ in \eqref{eq:intpsi}. 
In the parametrization with respect to $\bsy\in U$, 
$Z$ and $Z'$ take the form of infinite-dimensional, 
iterated integrals with respect to the uniform 
prior $\bpi_0(\dd\bsy)$.
%
\subsection{Well-posedness and approximation}
\label{sec:WellPosAppr}
For the computational viability of 
Bayesian inversion the quantity $\IE^{\bpi^\delta}[\phi]$
should be stable under perturbations of the data $\delta$ 
and under changes in the forward problem stemming, for example,
from discretizations as considered in Section \ref{sec:Discr}.
 
Unlike deterministic inverse problems 
where the data-to-solution maps can be 
severely ill-posed, for $\Gamma > 0$ 
the expectations \eqref{eq:intpsi} are Lipschitz continuous 
with respect to the data $\delta$, 
{\em provided that the potential $\Phi_\Gamma$ 
     in \eqref{eq:DefPhiu} is locally 
     Lipschitz with respect to the data $\delta$} 
     in the following sense.
\begin{assumption}\label{Ass:LipPhi}
\cite[Assumption 4.2]{Stuart13}
Let $\tilde{X}\subseteq X$ and assume 
$\Phi_\Gamma \in C(\tilde{X}\times Y;\IR)$
is Lipschitz on bounded sets. 
Assume also that there exist 
functions $\calM_i: \IR_+\times \IR_+\to \IR_+$, $i=1,2$,
(depending on $\Gamma > 0$)
which are monotone, non-decreasing separately in each argument,
and with $\calM_2$ strictly positive, such that for all $u\in \tilde{X}$, 
and for all $\delta,\delta_1,\delta_2\in B_Y(0,r)$
\begin{equation}\label{eq:PhiBddbelow}
\Phi_{\Gamma}(u;\delta) \geq -\calM_1(r,\| u \|_X),
\end{equation}
and
\begin{equation}\label{eq:PhiLocLip}
|\Phi_\Gamma (u;\delta_1) - \Phi_\Gamma (u;\delta_2) | 
\leq 
\calM_2(r, \| u \|_X) \| \delta_1 - \delta_2 \|_Y 
\;.
\end{equation}
\end{assumption}
We remark that in the present context,
\eqref{eq:PhiLocLip} follows from \eqref{eq:DefPhiu};
for convenient reference, we include \eqref{eq:PhiLocLip} 
in Assumption \ref{Ass:LipPhi}.
Under Assumption \ref{Ass:LipPhi}, 
the expectation \eqref{eq:intpsi} depends Lipschitz 
continuously on $\delta$ (see, e.g., \cite[Sec. 4.1]{Stuart13} for a proof):
\begin{equation} \label{eq:LipData}
\forall \phi\in L^2(\bpi^{\delta_1},X;\IR)\cap L^2(\bpi^{\delta_2},X;\IR):
\quad 
| \bbE^{\bpi^{\delta_1}}[\phi] -  \bbE^{\bpi^{\delta_2}}[\phi]|
\leq 
C(\Gamma,r) \| \delta_1 - \delta_2 \|_Y\;.
\end{equation}

Ahead, we shall be interested in the impact of approximation
errors in the forward response of the system (e.g.~due to 
discretization and approximate numerical solution of system responses)
on the Bayesian predictions \eqref{eq:intpsi}.
To ensure continuity of the expectations
\eqref{eq:intpsi} w.r.t.~changes in the potential, 
we impose the following assumption.
\begin{assumption}\label{Ass:ApprPhi}
\cite[Assumption 4.6]{Stuart13}
Let $\tilde{X}\subseteq X$ and assume 
$\Phi_\Gamma \in C(\tilde{X};\IR)$
is Lipschitz on bounded sets. 
Assume also that there exist functions 
$\calM_i: \IR_+ \to \IR_+$, $i=1,2$, 
independent of the number $M$ of degrees of freedom
in the discretization of the forward problem,
where the functions $\calM_i$ 
are monotonically non-decreasing separately in each argument,
and with $\calM_2$ strictly positive,
such that for all $u\in \tilde{X}$, and all $\delta \in B_Y(0,r)$,
%
\begin{equation}\label{eq:PhiBddbelow2}
\Phi_{\Gamma}(u;\delta) \geq - \calM_1(\| u \|_X),
\end{equation}
and 
there is a positive, monotonically decreasing $\varphi(\cdot)$
such that $\varphi(M)\to 0$ as $M\to \infty$, 
monotonically and uniformly w.r.t.~$u\in \tilde{X}$ (resp. w.r.t.~$\bsy \in U$)
and such that
\begin{equation}\label{eq:PhiAppr}
|\Phi_\Gamma(u;\delta) - \Phi^{M}_\Gamma (u;\delta) | 
\leq 
\calM_2(\| u \|_X) \varphi(M)
\;.
\end{equation}
\end{assumption}
Denote by $\bpi^\delta_{M}$ the Bayesian posterior, 
for given data $\delta\in Y$,
with respect to the approximate potential $\Phi^{M}_\Gamma$
obtained from a Petrov-Galerkin discretization \eqref{eq:parmOpEqh_trun}.
\begin{proposition}\label{prop:ErrBayesExpec}
Suppose that
Assumption \ref{Ass:ApprPhi} holds, and 
assume that for $\tilde{X} \subseteq X$ and for some bounded 
set $B\subset X$ we have $\bpi_0(\tilde{X} \cap B)>0$ and 
$$
X\ni u \mapsto \exp(\calM_1(\| u\|_X)) (1 +\calM_2^2(\| u\|_X) ) \in L^1_{\bpi_0}(X;\IR) 
\;.
$$
Then there holds, 
for every QoI $\phi:X\to \IR$ such that
$\phi\in  L^2_{\bpi^{\delta}}(X;\IR)\cap L^2_{\bpi^{\delta}_M}(X;\IR)$ 
uniformly w.r.t.~$M$ and such that $Z>0$ in \eqref{eq:Z},
the consistency error bound
\begin{equation}\label{eq:BayAppr}
\|\bbE^{\bpi^{\delta}}[\phi] -  \bbE^{\bpi^{\delta}_M}[\phi]\|_\IR
\leq 
C(\Gamma,r) 
\varphi(M)
\;.
\end{equation}
\end{proposition}
For a proof of Proposition \ref{prop:ErrBayesExpec}, 
we refer to \cite[Thm. 4.8, Rem. 4.9]{Stuart13}.

Below, we shall present concrete choices for the 
convergence rate function $\varphi(M)$ in estimate \eqref{eq:PhiAppr},
in terms of 
i) {\em ``dimension truncation''} of the 
   uncertainty parametrization \eqref{eq:uviapsi},
   i.e.~to a finite number of $s\geq 1$ terms in \eqref{eq:uviapsi}, 
   and 
ii) {\em discretization} of the dimensionally truncated problem
for particular classes of forward problems.
The verification of the consistency condition
\eqref{eq:PhiAppr} in either of these cases will be
based on the following observation.
\begin{proposition}\label{prop:PhiNCheck}
Assume we are given a sequence $\{ q_M \}_{M \geq 1}$
of approximations to the forward response 
$X\ni u\mapsto q(u) \in \bcX$
such that, with the parametrization \eqref{eq:uviapsi},
\begin{equation}\label{eq:qConsis}
\sup_{u\in \tilde{X}} \| (q-q_M)(\bsy) \|_{\bcX} \leq \varphi(M)
\end{equation}
with a consistency error bound $\varphi(M) \downarrow 0$ as 
in Assumption \ref{Ass:ApprPhi}.
Denote by $G^M$ the corresponding
(Galerkin) approximations of the parametric forward maps.
Then the approximate Bayesian potential is
\begin{equation}\label{eq:PhiN}
\Phi^M_{\Gamma}(u;\delta) 
= \frac{1}{2} (\delta - \cG^M(u))^\top \Gamma^{-1} (\delta - \cG^M(u))
: X\times Y \to \bbR \, ,
\end{equation}
where $\cG^M := \cO\circ G^M$, satisfies \eqref{eq:PhiAppr}.
\end{proposition}
\begin{proof}
By definition \eqref{eq:DefPhiu} of the Bayesian potential,
for every $u\in \tilde{X}$ (i.e.~every $\bsy\in U$ defined
in \eqref{eq:DefSetU}) %
\begin{align*}
    &
    |\Phi_{\Gamma}(u;\delta) - \Phi_{\Gamma}^M(u;\delta)| 
    \\
    &=
    \frac{1}{2}
    \left|
    \big(\delta-(\cO\circ G)(u)\big)^\top \Gamma^{-1} \big(\delta - (\cO\circ G)(u)\big)
    -
    \big(\delta-(\cO\circ G^M)(u)\big)^\top \Gamma^{-1} \big(\delta - (\cO\circ G^N)(u)\big)
    \right| \\
    &\leq
    \frac{1}{2}
    \big\| \Gamma^{-1/2} \cO \big\|_{\bcX^*}
    \big\| (G-G^M)(u) \big\|_{\bcX}
    \big|2\delta - \cO\circ(G+G^M)(u)\big|_\Gamma 
    \\
    &=
    \frac{1}{2}
    \big\| \Gamma^{-1/2} \cO \big\|_{\bcX^*}
    \big\| (q-q_M)(\bsy) \big\|_{\bcX}
    \big|2\delta - \cO\circ(q+q_M)(\bsy)\big|_\Gamma
    \;.
\end{align*}
%
\end{proof}
Throughout the following, we will denote by $q_M$ the 
Petrov-Galerkin discretization \eqref{eq:parmOpEqh_trun} 
in Section \ref{sec:Discr}, 
with $M=M_h={\rm dim}(\bcX^h) = {\rm dim}(\bcY^h)$
degrees of freedom.
\section{Holomorphic parameter dependence}
\label{sec:Hol}
As indicated, a key role in the present paper is
played by holomorphy of countably parametric families of
operator equations and their solution families.
By this we mean that the parametric family of solutions
permits, with respect to each parameter $y_j$, a holomorphic
extension into the complex domain $\C$; for purposes of QMC
integration, in addition some uniform bounds on these
holomorphic extensions must be satisfied in order to 
prove approximation rates and QMC quadrature error bounds
which are independent of the number of parameters which are
``activated'' in the approximation, i.e.~in
the QMC quadrature process.
In \cite{HaSc11,CCS2}, the notion of 
{\em $(\bsb, p, \eps)$-holomorphy of parametric solutions}
has been introduced to this end. %
\emph{%
In the remainder of Section \ref{sec:Hol} and 
throughout the next Section \ref{sec:anadepsol}, all
spaces $X$, $Y$, $\bcX$ and $\bcY$ will be understood as 
Banach spaces over $\C$, without notationally indicating so.
}
\subsection{Holomorphic Families of Countably-Parametric Maps}
\label{sec:HolFam}
\begin{definition} \label{def:peanalytic} ($(\bsb,p,\eps)$-holomorphy)
Let $\eps > 0$ and $0<p<1$ be given. 
For a positive sequence 
$\bsb = (b_j)_{j\geq 1} \in \ell^p(\IN)$, we say that a parametric solution family 
$q(\bsy) : U\to \cX$ of \eqref{eq:main}
satisfies the {\em $(\bsb,p,\eps)$-holomorphy assumption} 
if and only if all of the following conditions hold:
\begin{enumerate}
\item 
The map $\bsy\mapsto q(\bsy)$ from $U$ to $\cX$, for each $\bsy\in U$,
is uniformly bounded with respect to the parameter sequence $\bsy$, 
i.e.~there is a bound $C_0 > 0$ such that
\be
\label{ubNewProblem}
\sup_{\bsy\in U}\|q(\bsy)\|_X \leq C_0\;.
\ee
\item 
There exists a positive sequence 
$\bsb = (b_j)_{j\geq1} \in \ell^p(\N)$ 
such that, 
for any sequence $\bsrho:=(\rho_j)_{j\geq1}$
of numbers $\rho_j > 1$ that satisfies 
\be
\label{eq:rho_b}
\sum_{j \geq 1} (\rho_j-1) b_j \leq \eps,
\ee
for sufficiently small $\eps > 0$, 
the parametric solution map $U\ni\bsy\mapsto q(\bsy)$ 
admits an extension
$\bsz \mapsto q(\bsz)$ to the complex domain
that is holomorphic with
respect to each variable $z_j$ 
in a cylindrical set of the form 
${O}_{\bsrho} := \bigotimes_{j\geq 1} {O}_{\rho_j}$ 
where, for every $j\geq 1$, 
$[-1,1] \subset O_{\rho_j}\subset \C$ 
with open sets  $O_{\rho_j}\subset\C$.

\item
For any poly-radius $\bsrho$ satisfying \eqref{eq:rho_b}, 
there is a second family 
$\tilde{O}_{\bsrho} := \bigotimes_{j\geq 1} \tilde{O}_{\rho_j}$
of open, cylindrical sets 
$$
[-1,1] \subset {O}_{\rho_j} \subset \tilde{O}_{\rho_j} \subset \C
$$
(strict inclusions), such that the
extension is bounded on $\overline{\tilde{O}_\bsrho}$
according to
\be
\sup_{\bsz\in\tilde{O}_\rho}\|q(\bsz)\|_X \leq C_\eps \, ,
\ee
where the bounds $C_\eps > 0$ depend on $\eps$, 
but are independent of $\bsrho$.
\end{enumerate}
\end{definition}
The notion of $(\bsb, p,\eps)$-holomorphy depends implicitly on the 
choice of sets $\tilde{O}_{\rho_j}$.
Depending on the approximation process in the parameter 
domain $U$ under consideration, a particular choice of the sets 
$ \tilde{O}_{\rho_j}$ has to be made in order to obtain sharp
convergence bounds under minimal holomorphy requirements.

Some
particular choices of sets $\tilde{O}_{\rho_j}$ are 
as follows.
For a real number $\kappa > 1$, 
we denote by 
$\calD_\kappa = \{ z\in \bbC | |z| \leq \kappa \}$
the closed disc of radius $\kappa>1$ in $\bbC$.
Choosing 
$\tilde{O}_{\rho_j} = \calD_{\rho_j}$
implies, for example, $p$-summability of partial sums
of finitely truncated Taylor expansions of $q(\bsy)$ 
about $\bsy = 0$, cf. \cite{CCS2} and the references 
there. This choice arises also in connection with 
Chebysev approximation, cf. \cite{HaSc11}.
The existence of 
analytic continuations to polydiscs
arises naturally in the context of 
{\em affine parametric operator equations}
which were considered in \cite{DKGNS13}.

A second choice is the 
{\em Bernstein ellipse} in the complex plane,
which is defined for some $\kappa > 1$ by
(cp. \cite[Sec. 1.13]{Davis})
\be
\tilde{O}_\kappa
=
\cE_\kappa 
:= 
\bigg\{\frac {w+w^{-1}}{2}: w\in \C\;, 1 \leq |w| \leq \kappa \bigg\},
\ee
which has foci at $\pm 1$ and semi axes of length 
$\frac{1}{2}(\kappa + 1/\kappa)$  
and 
$\frac{1}{2}(\kappa - 1/\kappa)$.
This class of holomorphy domains is well known to afford sharp 
results on the convergence of approximations by 
truncated Legendre series, cf.  \cite{Davis,CCS2,CDS2}. 
Importantly, this class of domains can be obtained by
{\em local} analytic continuation into discs of radii
larger than $\rho_j-1$ of the parametric mapping $q(\bsy)$
about a finite number of points $y_m\in [-1,1]$ 
(cp. \cite{CCS2}).

For the derivative bounds which arise in connection
with higher order Quasi-Monte Carlo error analysis
(see, e.g., \cite{DKGNS13,KSS12}) 
we use the following family of smaller continuation domains,
as in \cite{DLGCS14}:
for $\kappa>1$, consider by $\cT_\kappa$ the set of points
\begin{equation}\label{eq:TubeDef}
\cT_\kappa 
= \{ z\in \bbC| {\rm dist}(z,[-1,1]) \leq \kappa - 1\}
= \bigcup_{-1\leq y \leq 1} \{ z\in \IC | |z-y|\leq \kappa -1 \} 
\subset \bbC
\;.
\end{equation}
Then, once more, for a poly-radius $\bsrho$ satisfying
\eqref{eq:rho_b}, we denote by $\cT_\bsrho$ the 
corresponding cylindrical set 
$\cT_\bsrho := \bigotimes_{j\geq 1} \cT_{\rho_j}\subset \C^\N$.
%
\subsection{Examples of $(\bsb, p,\eps)$-holomorphic families}
\label{ssec:affparops}
We present several classes of examples of 
parametric equations $\cR(u;q)$ in \eqref{eq:main} for
which $(\bsb, p,\eps)$-holomorphy can be verified.
To this end, we denote by $\bcX$ and $\bcY$ 
separable and reflexive Banach spaces over
$\mathbb{R}$ (all results will hold with the obvious modifications also
for spaces over $\mathbb{C}$) with (topological) duals $\bcX'$ and
$\bcY'$, respectively. 
By $\cL(\bcX,\bcY')$, we denote the set of 
bounded linear operators $A:\bcX \to\bcY'$.
%
We next consider parametric forward models and
the regularity of their (countably-) parametric solution
families.
The following result, proved in \cite{CCS2},
shows that holomorphy of the solution map $\bsy\mapsto q(\bsy)$ 
follows from the holomorphy of the maps 
$A$ and $F$ in \eqref{eq:main}. 
Ahead, in Section \ref{sec:pdTh}, it will be seen
to imply holomorphy of the parametric Bayesian posterior
which, in turn, will be seen in Section \ref{sec:QMC} 
to yield dimension independent convergence rates for 
higher order QMC quadrature approximations of
$Z$ and $Z'$ in the Bayesian estimate \eqref{eq:intpsi}.
\begin{theorem}
\label{thm:AnalyticLaxMilgram}
For $\eps > 0$ and $0< p <1$, assume that there exist
a positive sequence $\bsb = (b_j)_{j\geq1} \in \ell^p(\N)$, 
and two constants
$0<r \leq R < \infty$ 
independent of $u\in \tilde{X}$ such that the following holds:
\begin{enumerate} 
\item
For any sequence $\bsrho := (\rho_j)_{j\geq1}$
of numbers strictly greater than $1$ that satisfies 
\eqref{eq:rho_b},
the parametric maps $\fa(\bsy;\cdot,\cdot)$ 
corresponding to the linear operator $A(\bsy)\in \cL(\bcX,\bcY')$
and $F$ admit extensions to complex parameters that are holomorphic 
with respect to every variable $\bsz$ on a set of the form 
$O_{\rho} = \bigotimes_{j\geq 1} O_{\rho_j}$,
where $O_{\rho_j}\subset \C$ is an open set 
containing $\tilde{O}_{\rho_j}$. 
\item
These extensions satisfy for all $\bsz\in O_{\rho}$ the
uniform continuity conditions
\be
\label{contFB}
\sup_{w\in \bcY\backslash\{0\}} \frac {|F(\bsz;w)|}{\|w\|_\bcY} \leq R, 
\quad  
\sup_{v\in \cX \backslash\{0\},w\in \cY\backslash \{0\}} 
\frac { |\fa(\bsz;v,w)| }{\|v\|_\bcX \|w\|_\cY}\leq R,
\ee
and the uniform inf-sup conditions: 
there exists $r>0$ such that
for every $\bsz\in O_{\rho}$ there hold the 
uniform inf-sup conditions
\be\label{infsupFB}
\inf_{v\in \cX \backslash \{0\}}\sup_{w\in \cY \backslash \{0\}} 
\frac{|\fa(\bsz;v,w)|}{\|v\|_\cX\|w\|_\cY} 
\geq r
\quad\mbox{and} \quad 
\inf_{w\in \cY\backslash \{0\}}\sup_{v\in \cX\backslash \{0\}}
\frac{|\fa(\bsz;v,w)|}{\|v\|_\cX\|w\|_\cY} \ge r
\;.
\ee
\end{enumerate}
Then, the nonlinear, parametric residual 
operator 
$\cR(u(\bsz);q) = A(u(\bsz);q) - F(u(\bsz);q)$ in \eqref{eq:main} (where $\bsz \in O_{\rho}$)
satisfies the $(\bsb, p,\eps)$-holomorphy assumptions 
with the same $p$ and $\eps$ and with the same sequence $\bsb$.
\end{theorem}
We refer to Section \ref{sec:modelproblems} ahead for concrete examples.
\section{Higher order QMC-PG method for Bayesian inverse problems}
\subsection{Regularity of parametric solutions}
\label{sec:anadepsol}
The dependence of the solution $q(\bsy)$ of
the parametric, variational problem \eqref{eq:main} on the parameter
vector $\bsy$ is studied in this subsection. 
Precise bounds on the growth of the partial derivatives 
of $q(\bsy)$ with respect to $\bsy$ will be given. 
These bounds will, as in \cite{KSS12}, 
imply dimension independent convergence rates for QMC quadratures.

In the following, let $\N_0^\N$ denote the set of sequences $\bsnu =
(\nu_j)_{j\geq 1}$ of nonnegative integers $\nu_j$, and let $|\bsnu| :=
\sum_{j\geq 1} \nu_j$. For $|\bsnu|<\infty$, 
we denote the partial
derivative of order $\bsnu$ of $q(\bsy)$ 
with respect to $\bsy$ by
\be\label{eq:dqdy}
\partial^\bsnu_\bsy q(\bsy)
\,:=\,
\frac{\partial^{|\bsnu|}}{\partial^{\nu_1}_{y_1}\partial^{\nu_2}_{y_2}\cdots}q(\bsy)
\;.
\ee
In \cite{CDS1,KSS12,KunothCS2011}, bounds on the derivatives \eqref{eq:dqdy} 
were obtained by an induction argument which strongly relied on 
affine-parametric dependence of the parametric operator.

Alternative bounds on $\| (\partial^{\bsnu}_\bsy q)(\bsy)\|_{\cX}$
based on complex variable methods from \cite{CDS2,ScSt11,SS12,SS13,CCS2},
which give rise to product weights
at least for a finite (possibly large, but in general operator-dependent)
``leading'' dimension of the parameter space are derived in 
\cite{DLGCS14}. These bounds are based on a 
\emph{holomorphic extension of the parametric integrand functions to the complex domain} 
(we add that not all PDE problems afford such 
 extensions and refer to \cite{HoSc12Wave} for an example).

For the particular case of 
linear, countably affine-parametric operator families 
$(\bsb, p,\eps)$-holomorphy as in Definition~\ref{def:peanalytic} 
holds on polydiscs $\calD_\bsrho$.

We remark that the smaller tubes 
$\cT_\bsrho$ of holomorphy in \eqref{eq:TubeDef} 
are, nevertheless, important: on the one hand, 
the ensuing result about $(\bsb, p,\eps)$-holomorphic, 
countably parametric functions belonging to 
SPOD-weighted spaces of integrands 
admissible for QMC quadrature by higher order
digital nets is stronger (being valid under weaker hypotheses).
On the other hand, in certain cases 
the possibility of covering the parameter intervals $[-1,1]$
by a finite number of small balls (whose union is contained
in a tube $\cT_{\rho_j}$ for radius sufficiently close to 
$\rho_j = 1$) is crucial to verify holomorphy of the parametric
solution families for certain nonlinear operator equations, 
see for example \cite[Sec. 5.2]{CCS2}. 
The following result is the main result from \cite{DLGCS14};
it shows that $(\bsb, p,\eps)$-holomorphy with respect to the
domains $\cT_{\rho_j}\subset \bbC$ for $j\geq 1$ 
implies derivative bounds for
higher order QMC integration with dimension-independent rates
of convergence.
\begin{theorem}\label{thm:DsiboundC}
For every mapping $q(\bsy):U\to \cX$ 
which is $(\bsb, p,\eps)$-holomorphic
on a polytube $\cT_{\bsrho}$ of
poly-radius $\bsrho = (\rho_j)_{j\geq 1}$ with 
$\rho_j > 1$ satisfying \eqref{eq:rho_b},
there exists a sequence $\bsbeta\in \ell^p(\bbN)$ 
(depending on the sequence $\bsb$ in \eqref{eq:rho_b})
and a partition $\bbN = E \cup E^c$ such that 
the parametric solution $q(\bsy)$ satisfies,
for every $\bsnu\in \bbN_0^\bbN$ with $|\bsnu|<\infty$,
the bound
\be \label{eq:HybBd}
\sup_{\bsy \in U}
\| (\partial^\nu_\bsy q)(\bsy) \|_\cX
\leq
C_\eps
\bsnu_E! 
\prod_{j\in E}\beta_j^{\nu_j}  
\times 
|\bsnu_{E^c}|! 
\prod_{j\in E^c}\beta_j^{\nu_j}
\;.
\ee
Here, $E = \{1,2,...,J\}$ for some $J = J(\bsb) < \infty$
depending on the sequence $\bsb$ in \eqref{eq:rho_b}, 
and for $\bsnu\in \bbN_0^\bbN$,
we set $\bsnu_E := \{ \nu_j : j\in E \}$.
The sequence 
$\bsbeta = (\beta_j)_{j\geq 1}$ satisfies
$\beta_j = 4\| \bsb \|_{\ell^1(\bbN)}/\eps$, 
i.e.~it
is in particular independent of $j$ for $1\leq j \leq J$.
Moreover, 
$\beta_j \lesssim b_j $ for $j>J$ with the implied
constant depending only on $J(\bsb)$ and on $\| \bsb \|_{\ell^1(\bbN)}$.
\end{theorem}
The derivative bounds \eqref{eq:HybBd} were deduced from
$(\bsb,p,\eps)$-holomorphy of the forward map $U\ni \bsy \mapsto q(\bsy)$.
The same argument immediately implies corresponding 
derivative estimates for bounded, linear observation functionals
$\cO(\cdot)$ of the parametric solution.
With the $(\bsb,p,\eps)$-holomorphy
of the posterior densities $\Theta(\bsy)$ and $\Psi(\bsy)$ in 
Proposition~\ref{thm:pdTheta}, corresponding bounds for 
the parametric integrand functions in the Bayesian estimate
\eqref{eq:intpsi} follow analogously. 
We sum up these observations in the following.
\begin{corollary}\label{coro:DsiBoundThPsi}
Assume that the parametric solution map $U\ni \bsy \mapsto q(\bsy)$
of the forward problem is $(\bsb,p,\eps)$-holomorphic. 
Then, for every bounded, linear observation functional 
$\cO(\cdot): \bcX\to Y$, the 
countably-parametric uncertainty-to-observation map
$G(\bsy) := \cO(q(\bsy)): U \to Y$ 
satisfies the estimates \eqref{eq:HybBd}:
for given $0<\eps < 1$ and $\bsb\in \ell^p(\IN)$
there exist 
a constant $C_\eps > 0 $, 
a sequence $\bsbeta(\eps)\in \ell^p(\IN)$ and 
a partition $\IN = E\cup {E^c}$ depending only on $\eps$ and on $\bsb$
such that
for every $\bsnu\in \bbN_0^\bbN$ with $| \bsnu | < \infty$,
there holds
\begin{equation}\label{eq:FsPODBd}
\sup_{\bsy\in U}
\| \partial^{\bsnu}_{\bsy} G (\bsy) \|_Y 
\leq 
C_\eps
\bsnu_E!
\prod_{j\in E}\beta_j^{\nu_j}
\times
|\bsnu_{E^c}|!
\prod_{j\in {E^c}}\beta_j^{\nu_j}
\leq 
C_\eps |\bsnu|! \bsbeta^\bsnu
\;.
\end{equation}
\end{corollary}
%
\subsection{Holomorphy of Bayesian posterior $\Theta(\bsy)$}
\label{sec:pdTh}
Our verification of existence and $(\bsb, p,\eps)$-holomorphy of 
analytic continuations $\Theta(\bsz)$ and $\Psi(\bsz)$
defined in \eqref{eq:PostDens} and \eqref{eq:psi}, 
respectively,
which appear in the parametric version \eqref{eq:intpsi} 
of Bayes' formula
will be based on $(\bsb, p,\eps)$-holomorphy
of the forward map $\bsy\mapsto q(\bsy)$ 
for the parametric posterior density $\Theta(\bsy)$
defined in \eqref{eq:post} and \eqref{eq:PostDens}.
Sufficient conditions for this were given in 
Theorem \ref{thm:AnalyticLaxMilgram} 
above.
\begin{proposition}\label{thm:pdTheta} \cite[Thm. 4.1]{SS13}
Consider the Bayesian inversion of the parametric
operator equation \eqref{eq:main} with uncertain input $u\in X$,
parametrized by the sequence 
$\bsy = (y_j)_{j\in \bbJ}\in U$.
Assume further that the corresponding
forward solution map $U\ni \bsy \mapsto q(\bsy)$ is
$(\bsb, p,\eps)$-holomorphic for some sequence 
$\bsb \in \ell^p(\IN)$ for some $0<p<1$ and some $\eps > 0$.

Then the Bayesian posterior densities 
$\Theta(\bsy)$ and $\Psi(\bsy)$
defined in \eqref{eq:PostDens} and \eqref{eq:psi}, respectively,
are, as a function of the parameter $\bsy$, likewise 
$(\bsb, p,\eps)$-holomorphic, 
{\em with the same $p$ and the same $\eps$}.

The modulus of the holomorphic extension of the
Bayesian posterior $\Theta(\bsy)$ over the polyellipse 
$\cE_\bsrho$ for a $(\bsb,\eps)$-admissible 
poly-radius $\bsrho$ as in \eqref{eq:rho_b} is bounded as
\begin{equation}\label{eq:ModBound}
B_\eps
=
\sup_{\bsz \in \partial \cE_\bsrho} |\Theta(\bsz)| 
\leq 
C\exp(\theta^2\| \Gamma^{-1}\|),
\end{equation}
with $\Gamma > 0$ denoting the positive definite 
covariance matrix in the additive,
Gaussian observation noise model \eqref{eq:DatDelta}.
The constants $\theta , C > 0$ in \eqref{eq:ModBound} 
depend on the condition number of the uncertainty-to-observation map 
$\cG(\cdot) = (\cO\circ G)(\cdot)$ 
but are independent of $\Gamma$ in \eqref{eq:DatDelta}.
The densities 
$\Theta(\bsy):U\to \IR$ and $\Psi(\bsy):U\to \IR$
in the Bayesian estimate \eqref{eq:intpsi} also satisfy 
estimates \eqref{eq:HybBd}, 
with norms taken in the respective spaces $\IR$ and $\IR$.
\end{proposition}
\subsection{Higher order quasi-Monte Carlo integration}
\label{sec:QMC}
In order to prove error bounds using QMC quadrature, 
we require the integrand to be smooth. 
A result on the smoothness of the $(\bsb,p,\eps)$-holomorphic 
solution of families of nonlinear parametric operator equations
with $(\bsb,p,\eps)$-holomorphic operators has been shown in \cite{DLGCS14} and 
is restated above (Theorem \ref{thm:DsiboundC}). 
For our purposes it is important to obtain error bounds which are independent 
of the dimension $s$, where $s$ is the {\em truncation dimension},
i.e.~we consider \eqref{eq:mainstrunc}
and its (Petrov-)Galerkin discretization \eqref{eq:parmOpEqh_trun}.
This means that we truncate the infinite sum in \eqref{eq:uviapsi} 
to a finite number $s \ge 1$ of terms and then estimate the resulting $s$ 
dimensional integral to approximate the mathematical expectation 
of the random solutions in \eqref{eq:Z} and \eqref{eq:intpsi}.
%

For a real-valued integrand function $g \in C^0([0,1]^s;\IR)$,
we consider the $s$-variate integration problem
\begin{equation}\label{eq:IsF}
 I_s(g) \,:=\,
 \int_{[0,1]^s} g(\bsy) \,\rd\bsy \;.
\end{equation}
We approximate this integral by an equal weight QMC quadrature rule
\begin{equation}\label{eq:QNs}
  Q_{N,s}(g) \,:=\,
  \frac{1}{N} \sum_{n=0}^{N-1} g(\bsy_n)\;,
\end{equation}
where the quadrature points $\bsy_0,\ldots,\bsy_{N-1} \in [0,1]^s$ are judiciously chosen. 
In the following we restate the necessary definitions 
and results from \cite{DKGNS13}.
\begin{definition} \label{def_F_norm}
Let nonnegative integers $\alpha, s\in\bbN$, 
and real numbers $1\le q \le \infty$ and $1\le r \le \infty$ be given. 
Let $\bsgamma = (\gamma_\setu)_{\setu\subset\bbN}$ 
be a collection of nonnegative real numbers called \emph{weights}. 
For every $s\in \IN$, assume that the integrand function
$g: [0,1]^s \to \mathbb{R}$ has partial derivatives of orders up to 
$\alpha$ with respect to each variable. 
Define $0/0 := 0$ and $a/0 := \infty$ for $a > 0$.
The smoothness of the integrand function $g$ in \eqref{eq:IsF} 
is quantified by the unanchored Sobolev norm
\begin{equation}\label{eq:defFabs}
\begin{array}{rl}
\|g\|_{s,\alpha,\bsgamma,q,r}
& \displaystyle
:=
 \Bigg( \sum_{\setu\subseteq\{1:s\}} \Bigg( \gamma_\setu^{-q}
 \sum_{\setv\subseteq\setu} \sum_{\bstau_{\setu\setminus\setv} \in \{1:\alpha\}^{|\setu\setminus\setv|}} \\
&\qquad\qquad\quad 
 \displaystyle
 \int_{[0,1]^{|\setv|}} \bigg|\int_{[0,1]^{s-|\setv|}} \!
 (\partial^{(\bsalpha_\setv,\bstau_{\setu\setminus\setv},\bszero)}_\bsy g)(\bsy) \,
\rd \bsy_{\{1:s\} \setminus\setv}
 \bigg|^q \rd \bsy_\setv \Bigg)^{r/q} \Bigg)^{1/r},
\end{array}
\end{equation}
with the obvious modifications if $q$ or $r$ is infinite. Here $\{1:s\}$
is a shorthand notation for the set $\{1,2,\ldots,s\}$, and
$(\bsalpha_\setv,\bstau_{\setu\setminus\setv},\bszero)$ denotes a sequence
$\bsnu$ with $\nu_j = \alpha$ for $j\in\setv$, $\nu_j = \tau_j$ for
$j\in\setu\setminus\setv$, and $\nu_j = 0$ for $j\notin\setu$.
Let $\calW_{s,\alpha,\bsgamma,q,r}$ denote the Banach 
space of all such functions $g$ with finite norm.
\end{definition}

Note that if $\gamma_\setu = 0$ for some set $\setu$ then the corresponding projection term $$ \sum_{\setv\subseteq\setu} \sum_{\bstau_{\setu\setminus\setv} \in \{1:\alpha\}^{|\setu\setminus\setv|}}
       \int_{[0,1]^{s-|\setv|}}
         (\partial^{(\bsalpha_\setv,\bstau_{\setu\setminus\setv},\bszero)}_\bsy g)(\bsy) \,\rd \bsy_{\{1:s\} 
        \setminus\setv} $$ 
has to be $0$ for all $g \in \calW_{s, \alpha, \bsgamma, q, r}$.
The next theorem, from  \cite[Thm.~3.5]{DKGNS13},
states an upper bound on the worst-case integration error 
in $\calW_{s, \alpha,\bsgamma, q, r}$ using a QMC rule based on a digital net.
\begin{proposition} \label{thm:wce}
Let nonnegative integers $\alpha, s\in\bbN$ with $\alpha>1$, and real numbers $1\le q\le \infty$ and $1\le r\le \infty$ be given. Let $\bsgamma = (\gamma_\setu)_{\setu\subset\bbN}$ be a collection of
weights. Let $r'\ge 1$ be the H\"older conjugate of $r$, i.e.~$1/r + 1/r' = 1$. Let $b$ be
prime, $m\in\bbN$, and let $\calS=\{\bsy_n\}_{n=0}^{b^m-1}$ denote a
digital net with generating matrices 
$C_1,\ldots,C_s\in\bbZ_b^{\alpha m\times m}$. Then 
\[
  \sup_{\|g\|_{s,\alpha,\bsgamma,q,r} \le 1}
  \left| \frac{1}{b^m} \sum_{n=0}^{b^m-1} g(\bsy_n) - \int_{[0,1]^s} g(\bsy) \,\mathrm{d} \bsy \right|
  \,\le\, e_{s,\alpha,\bsgamma,r'}(\calS)\;,
\]
with
\begin{align} \label{def-B}
  e_{s,\alpha,\bsgamma,r'}(\calS)
  \,:=\, \Bigg(\sum_{\emptyset \neq \setu \subseteq \{1:s\}}
  \bigg(C_{\alpha,b}^{|\setu|}\, \gamma_\setu \sum_{\bsk_\setu \in \setD_\setu^*}
  b^{-\mu_{\alpha}(\bsk_\setu)} \bigg)^{r'} \Bigg)^{1/r'}\;.
      \end{align}
Here $\setD_\setu^*$ is the ``dual net without $0$ components'' projected
to the components in $\setu$.
Moreover, we have $\mu_{\alpha}(\bsk_\setu) = \sum_{j\in\setu}
\mu_\alpha(k_j)$ with
\begin{equation} \label{def-mu-k}
  \mu_\alpha(k)
  \,:=\,
  \begin{cases}
  0 & \mbox{\,if } k = 0, \\
  a_1 + \cdots + a_{\min(\alpha,\rho)} & 
  \begin{aligned}
   \mbox{if }
  k &= \kappa_1 b^{a_1-1} + \cdots + \kappa_\rho b^{a_\rho-1} \mbox{ with} \\
    &\kappa_i\in \{1,\ldots,b-1\} \mbox{ and } a_1>\cdots>a_\rho>0,
  \end{aligned}
  \end{cases}
\end{equation}
and
\begin{align}\label{eq:Cab}
 &C_{\alpha,b}
 \,:=\,
 \max\left(\frac{2}{(2\sin\frac{\pi}{b})^{\alpha}},\max_{1\le z\le\alpha-1}
 \frac{1}{(2\sin\frac{\pi}{b})^z}\right)
 \nonumber
 \\
 &\qquad\qquad\qquad\times
 \left(1+\frac{1}{b}+\frac{1}{b(b+1)}\right)^{\alpha-2}
 \left(3 + \frac{2}{b} + \frac{2b+1}{b-1} \right)\;.
\end{align}
\end{proposition}
We are interested in the case where the integrand 
$g(\bsy)$ is a composition of a continuous, linear functional $\Oc(\cdot) \in \bcX'$
with the (Petrov-)Galerkin approximation $q^h_s(2 \bsy-\bsone)$ 
of the dimension-truncated, parametric and
$(\bsb,p,\eps)$-holomorphic, operator equation \eqref{eq:NonOpEqn}.  
For every $s\in N$,
the dimension-truncated integrand functions 
$g(\bsy) := (\Oc \circ q_s)(\bsy_{\{1:s\}})$ are 
$(\bsb,p,\eps)$-holomorphic {\em uniformly w.r.t. $s\in \IN$}
(see Section \ref{sec:dimtrunc}).
It follows from Theorem~\ref{thm:DsiboundC} that they 
satisfy the derivative estimates \eqref{eq:HybBd}
uniformly w.r.t. $s\in \IN$.
In \cite[Sec.~3]{DKGNS13} and \cite[Prop. 4.1]{DLGCS14} 
we showed the following result on 
convergence rates of QMC quadratures
based on higher order digital nets for 
functions $g(\bsy)$ which satisfy \eqref{eq:HybBd}. 
\begin{proposition}\label{prop:main1}
Let $s\ge 1$ and $N = b^m$ for $m\ge 1$ be integers and $b$ be prime. 
Let $\bsbeta = (\beta_j)_{j\ge 1}$ be a sequence of positive numbers, 
and denote by $\bsbeta_s = (\beta_j)_{1\le j \le s}$ its truncation after $s$ terms. 
Assume that
\begin{equation} \label{p-sum}
  \exists\, 0<p\le 1 : \quad \sum_{j=1}^\infty \beta_j^p < \infty\;.
\end{equation}
Define, for $0<p\leq 1$ as in \eqref{p-sum},
\begin{equation} \label{alpha}
  \alpha \,:=\, \lfloor 1/p \rfloor +1 \;.
\end{equation}
Consider integrand functions $g(\bsy)$ 
whose mixed partial derivatives 
of order $\alpha$ satisfy
\begin{equation} \label{eq:like-norm}
\forall\, \bsy\in U \;\forall s\in \IN \;
\forall\, \bsnu \in \{0, 1, \ldots, \alpha\}^s:
\quad
 | (\partial^{\bsnu}_\bsy g)(\bsy)| 
\,\le\, 
c(g) \bsnu_E! 
\prod_{j\in E}\beta_j^{\nu_j}  
\times 
|\bsnu_{E^c}|! 
\prod_{j\in {E^c}}\beta_j^{\nu_j}
\end{equation}
for some fixed integer $J\in \mathbb{N}$ 
where $E = \{1, 2, \ldots, J\}$ and ${E^c} = \mathbb{N} \setminus E$,
and where $c(g)>0$ is independent of $\bsy$, $s$ and of $\bsnu$.
Then, for every $N\in \N$, one can construct 
an interlaced polynomial lattice rule of
order $\alpha$ with $N$ points using a fast
component-by-component algorithm, using
$\calO(\alpha \left(\min\{s, J\} + \alpha (s-J)_+ \right) N \log N)$ 
operations, plus $\calO(\alpha^2 (s-J)_+^2 N)$ update cost, plus 
$\calO(N + \alpha (s- J)_+ N)$ memory cost, where $(w)_+ = \max\{0, w\}$, 
such that there holds the error bound
\begin{equation}\label{qmc_upper_bound}
\forall s,N \in \N:\quad 
  |I_s(g) - Q_{N,s}(g)|
  \,\le\, C_{\alpha,\bsbeta,b,p}\, N^{-1/p} \;,
\end{equation}
where $C_{\alpha,\bsbeta,b,p} < \infty$ is a constant independent of $s$
and $N$.
\end{proposition}
\begin{remark}
It has been observed, see for instance \cite{GaCS14}, 
that in implementations of the CBC construction of good generating 
vectors components tend to repeat, 
resulting in rather nonuniform lower dimensional projections. 
To avoid this problem, \cite{GaCS14} used a ``pruned'' CBC algorithm 
where components of the generating vector are forced to differ from each other. 
A theoretical justification of this algorithm was provided in \cite{DK15}. 
If we use this modified algorithm in Proposition~\ref{prop:main1} 
then, provided that $N > 2s$, 
the upper bound \eqref{qmc_upper_bound} still remains valid albeit 
with a slightly larger constant:
the constant $C_{\alpha, b}$ is replaced by $2 C_{\alpha, b}$, 
which is still independent of the dimension $s$.
\end{remark}
\begin{remark}\label{rem_interval_change}
The bound \eqref{eq:like-norm} in Theorem~\ref{thm:DsiboundC} 
was shown for functions defined on the domain $[-1,1]^{\mathbb{N}}$. 
However, QMC theory uses the domain $[0,1]^s$. 
The change from $[-1,1]$ to $[0,1]$ can be achieved by the simple 
linear transformation $y \mapsto (y+1)/2$. 
Using \eqref{eq:HybBd} together with this change of variable in Proposition~\ref{prop:main1} 
increases the constant in \eqref{eq:Cab} by a factor of at most $2^\alpha$. 
Thus, in order for the theory to apply to the integrands from 
Sections~\ref{sec:HolOpEq} and \ref{sec:anadepsol}, 
we need to scale the Walsh constant $C_{\alpha, b}$ in \eqref{eq:Cab} 
by a factor $2^\alpha$. 
\end{remark}

If the function $g$ satisfies \eqref{eq:like-norm}, 
then its norm \eqref{eq:defFabs} with $r=\infty$ and for any $q$, 
can be bounded by
\begin{align*}
 \|g\|_{s,\alpha,\bsgamma,q,\infty}
 &\,\le\, c
 \max_{\setu\subseteq\{1:s\}}
 \gamma_\setu^{-1}
 \sum_{\bsnu_\setu \in \{1:\alpha\}^{|\setu|}}
 \bsnu_{\setu \cap E}!\,
 \prod_{j\in\setu\cap E} \left(2^{\delta(\nu_j,\alpha)}\beta_j^{\nu_j}\right)\;
 |\bsnu_{\setu \cap {E^c}}|!\,
 \prod_{j\in\setu \cap {E^c}} \left(2^{\delta(\nu_j,\alpha)}\beta_j^{\nu_j}\right)\;\\
 &\,=\, c(g)
  \max_{\setu\subseteq\{1:s\}}
 \gamma_\setu^{-1}
 \sum_{\bsnu_\setu \in \{1:\alpha\}^{|\setu|}}
 \bsnu_{\setu \cap E}!\,
 |\bsnu_{\setu \cap {E^c}}|!\,
 \prod_{j\in\setu} \left(2^{\delta(\nu_j,\alpha)}\beta_j^{\nu_j}\right)\;,
\end{align*}
where $\delta(\nu_j,\alpha)$ is $1$ if $\nu_j=\alpha$ and is $0$
otherwise. To make  $\|g\|_{s,\alpha,\bsgamma,q,\infty} \le c$, 
we choose
\begin{equation}\label{equ:hybridWeight}
 \gamma_\setu :=  \sum_{\bsnu_\setu \in \{1:\alpha\}^{|\setu|}}
 \bsnu_{\setu \cap E}!\,
 |\bsnu_{\setu \cap {E^c}}|!\,
 \prod_{j\in\setu} \left(2^{\delta(\nu_j,\alpha)}\beta_j^{\nu_j}\right).
\end{equation}
\subsection{Combined QMC-PG Error Bound}
\label{sec:CombErrBd}
We approximate the Bayesian estimate 
$\IE^{\pi^\delta}[\phi] = Z'/Z$ with 
the ratio estimator 
$Z'_{N,s,h}/Z_{N,s,h}$ 
where 
\begin{align}
Z'_{N,s,h} &= Q_{N,s} ( \Theta^{s,h}(\cdot) \phi( q^{s,h}(\cdot) )  \in \IR
\;, 
\label{ZpNsh} 
\\
 Z_{N,s,h}    &= Q_{N,s} (\Theta^{s,h}(\cdot) )  \in \IR
\;,
\label{ZNs} 
\end{align}
where $\Theta^{s,h}$ is given by (cf.\eqref{eq:PostDens} and 
\eqref{eq:PhiN})
\begin{equation}
\Theta^{s,h}(\bsy) 
= 
\exp\bigl(-\Phi^{M}_\Gamma(u;\delta)\bigr)\Bigl|_{u=\langle u \rangle 
+ \sum_{j=1}^s y_j\psi_j},
\end{equation}
and  
$q^{s,h}(\bsy) := q^h(\bsy_{\{1:s\}})$ with 
$q^h(\bsy_{\{1:s\}})$ being the Galerkin approximation of
the forward problem with dimension-truncation to dimension $s$
as defined in \eqref{eq:parmOpEqh_trun}.
We recall that $M=M_h={\rm dim}(\bcX^h) = {\rm dim}(\bcY^h)$ (see Section \ref{sec:Discr}).
We are now in position to prove our main result:
an error bound for the approximate Bayesian estimates
which accounts for the errors of dimension truncation
to finite dimension $s$ given by \eqref{eq:mainstrunc}, 
Petrov-Galerkin discretization
\eqref{eq:parmOpEqh_trun} of \eqref{eq:mainstrunc}
and, finally, the QMC integration \eqref{eq:QNs} of the 
goal functional $G(\cdot)$ evaluated at the dimensionally
truncated solution, i.e.~of the approximation
$Q_{N,s}(G(q^{s,h}))$ of $I(G(q))$. 
All implied constants in the error bounds 
of the ensuing theorem are
independent of the truncation dimension $s$.
\begin{theorem}\label{thm:QMCPGErrBd}
Suppose that 
Assumptions \ref{asmp:LocSolv}, \ref{Ass:LipPhi} and \ref{Ass:ApprPhi}
hold. Suppose further that the parametric forward
problem \eqref{eq:NonOpEqn} admits, for every $\bsy\in U$, 
a unique solution $q(\bsy)$ which belongs to the smoothness
space $\cX_t$ for some $t>0$ affording the approximation
property \eqref{eq:apprprop}. Assume further the 
sparsity condition \eqref{eq:psumpsi} for the uncertainty
parametrization, and that conditions \eqref{eq:ordered} hold.
Then,
for given positive definite observation noise covariance 
$\Gamma$, the normalization constant $Z$ in \eqref{eq:Z} is positive.

Assume further that the QoI $\phi:\bcX \to \IR$ is Lipschitz
continuous in an open ball in $\bcX$ 
about the forward solution $q(\langle u \rangle)$ at the
nominal input $\langle u \rangle \in X$ 
(resp. about the origin $\bsnull\in U$).

Then there exist $N_0,s_0,h_0$ (depending on 
the bound $r>0$ on the size $|\delta|$ of the data,
and on the observation noise 
covariance $\Gamma$ in Assumption \ref{Ass:LipPhi}),
such that
the QMC-PG approximation of the Bayesian estimate
$\IE^{\pi^\delta}[\phi]$ in \eqref{eq:intpsi} 
obtained as approximate ratio estimate
$Z'_{N,s,h}/Z_{N,s,h}$ with
QMC-PG approximations $Z'_{N,s,h}$ and $Z_{N,s,h}$
of the integrals $Z'$ and $Z$ in \eqref{eq:intpsi},
satisfies, for $N\geq N_0$, $s\geq s_0$ and for $h\leq h_0$, 
the error bound
\be\label{eq:RatEstErrBd}
\left| \IE^{\pi^\delta}[\phi] - Z'_{N,s,h}/Z_{N,s,h} \right|
\leq 
C(\Gamma,p) \left(h^t + s^{-(1/p-1)} + N^{-1/p} \right)
\;.
\ee
%
\end{theorem}
Before giving the proof, we remark that the constant 
$C(\Gamma,p)$ depends on the observation noise
covariance $\Gamma > 0$ and on the summability exponent $p$,
but is independent of the truncation dimension $s$ 
and of the number $N$ of QMC points. 
We refer to \cite{SS14} for details.
We also remark that the ensuing proof does not require
bounds \eqref{eq:like-norm} for (functionals of) 
the Petrov-Galerkin discretized solution and, therefore, 
the verification of $h$-independence of the domain of holomorphy of the PG
approximation $q^{s,h}(\bsy)$ of the parametric forward solution
is not necessary.
\begin{proof} (of Theorem \ref{thm:QMCPGErrBd}):
Both integrands $\Psi(\bsy)$ and $\Theta(\bsy)$ in 
the definition \eqref{eq:intpsi} of $Z$ and of $Z'$
are $(\bsb,p,\eps)$-holomorphic by Proposition \ref{thm:pdTheta}.
By Theorem \ref{thm:DsiboundC}, 
their approximations $\Psi^{s,h}(\bsy)$ and $\Theta^{s,h}(\bsy)$
satisfy the derivative estimates \eqref{eq:HybBd} and \eqref{eq:FsPODBd},
uniformly with respect to the truncation dimension $s$, and the 
Petrov-Galerkin discretization parameter $h$: 
as dimension truncation amounts to restricting $\bsy$ to $\bsy_{\{ 1: s\}}$
(see Section \ref{sec:dimtrunc}), the QMC error bounds in 
Proposition~\ref{prop:main1} are applicable.

By the triangle inequality, we have
\begin{equation}\label{triangle}
\begin{aligned}
|Z' Z_{N,s,h} - Z'_{N,s,h} Z| &= | Z'Z_{N,s,h} - ZZ' + ZZ' - Z'_{N,s,h} Z | 
\\
& \le |Z'| |Z_{N,s,h} - Z| + |Z| |Z' - Z'_{N,s,h}| 
\;.
\end{aligned}
\end{equation}
By condition \eqref{eq:Z}, 
there is $N_0, s_0 \in \IN$ and $h_0>0$
so that $Z_{N,s,h} \ge Z/2$
for all $N\geq N_0$, $s\geq s_0$ and $0<h\leq h_0$.
Using this and \eqref{triangle} we obtain
\begin{align*}
\left| \frac{Z'}{Z} - \frac{Z'_{N,s,h}}{Z_{N,s,h}} \right|
 &  = \frac{|Z'Z_{N,s,h} - Z'_{N,s,h} Z |}{ |Z Z_{N,s}|} 
 \le \frac{ |Z'| |Z_{N,s,h} - Z|}{|Z Z_{N,s,h}|} 
    + \frac{|Z'-Z'_{N,s,h}|}{|Z_{N,s,h}|} 
\\
 & \le  \frac{2 |Z'| |Z_{N,s,h} - Z|}{|Z|^2} + \frac{2 |Z'-Z'_{N,s,h}|}{|Z|} 
\;.
\end{align*}
With $Z = I(\Theta)$ and  $Z_{N,s,h} = Q_{N,s}(\Theta^{s,h})$, 
we bound the error $|Z - Z_{N,s,h}|$ using Proposition~\ref{prop:main1}.
There exists a constant $C>0$ (independent of the parameter dimension $s$) 
such that for all $N$
%
\begin{align*}
|I(\Theta) - Q_{N,s}(\Theta)| \le C N^{-1/p}
\;.
\end{align*}
With $Z' = I(\Theta \phi(q))$ and $Z'_{N,s,h}$ as in \eqref{ZpNsh},
we bound the combined error $|Z' - Z'_{N,s,h}|$ 
incurred by dimension-truncation, 
QMC integration and Petrov-Galerkin discretization as follows:
$$
\begin{array}{rcl}
\left| I(\Theta \phi(q)) - Q_{s,N}( \Theta \phi(q^{s,h})) \right|
& \leq  & \displaystyle 
\| \Theta \|_\infty \big( \left| I(\phi(q)) - I(\phi(q^s)) \right|
+  
\left| I(\phi(q^s)) - Q_{N,s}(\phi(q^s)) \right| 
\\
& + & \displaystyle 
\left| Q_{N,s}(\phi(q^s)) - Q_{N,s}(\phi(q^{s,h})) \right| \big)
\;.
\end{array}
$$
Here, the first term can be bounded using Proposition~\ref{prop:trunc},
and \eqref{eq:Idimtrunc} and \eqref{eq:DTbound} 
yield
$$
\left| I(\phi(q)) - I(\phi(q^s)) \right|
= 
\left| I(\phi(q)) - I_s(\phi(q)) \right|
\leq 
C(p) 
\| \phi \|_{\cL(\cX,\IR)} s^{-(1/p-1)} 
\;.
$$
The second term is bounded using Proposition~\ref{prop:main1} 
as
$$
\left| I(\phi(q^s)) - Q_{N,s}(\phi(q^s)) \right|
=
\left| I_s(\phi(q)) - Q_{N,s}(\phi(q^s)) \right|
\leq 
CN^{-1/p} 
\;.
$$
The third term is estimated by the PG-discretization error
$$
\left| Q_{N,s}(\phi(q^s)) - Q_{N,s}(\phi(q^{s,h})) \right|
=
\left| Q_{N,s}(\phi(q^s - q^{s,h})) \right|
\leq 
\| \phi \|_{\cL(\cX,\IR)} 
\sup_{\bsy\in U} \| q^s(\bsy) - q^{s,h}(\bsy) \|_\cX \;,
$$
which, in turn, is estimated by Proposition~\ref{prop:stab} 
as 
$$
\left| Q_{N,s}(\phi(q^s)) - Q_{N,s}(\phi(q^{s,h})) \right|
\leq 
\frac{C}{\bar{\mu}}  \| \phi \|_{\cL(\cX,\IR)} h^t \sup_{\bsy\in U} \| q(\bsy) \|_{\bcX_t} 
\;.
$$
\end{proof}
\begin{remark}\label{rmk:Duality}
The preceding analysis used that $\phi\in \cL(\bcX,\IR)= \bcX'$. 
If the QoI $\phi$ has more regularity, such as $\phi \in \cL(\bcX_{t'},\IR)$ 
for some $t' > 0$, higher PG convergence rates $O(h^{t+t'})$ 
in \eqref{eq:RatEstErrBd} are possible, provided that the 
adjoint of the differential $ (D_q\cR)(u;q_0) $ 
is boundedly invertible in suitable scales of spaces.
We refer to \cite{KSS12,KSS13,DKGS14} for a statement of 
results and proofs in the case of linear, affine-parametric forward problems.
\end{remark}
%
\section{Model Problems}\label{sec:modelproblems}
We illustrate the hypotheses in Section~\ref{sec:HolOpEq},
with a view towards the numerical experiments in 
Section~\ref{sec:results} ahead, by a model, 
affine-parametric diffusion
problem which was already considered in \cite{CDS2}.
We emphasize that these model problems are selected
to illustrate the preceding error analysis; 
the preceding theory is applicable to considerably more general problems.
\subsection{Affine-parametric diffusion problem}
\label{sec:DiffEq}
To illustrate the preceding convergence estimates,
we consider a linear  elliptic diffusion equation on 
the physical domain $D=(0,1)^d$ in space dimension $d=1$ or $d=2$,
with uncertain diffusion coefficient $u(x)$, 
known source term $f(x)\in L^2(D)$ 
and with homogeneous Dirichlet boundary conditions, ie.,
\be \label{eq:diff:diffeq}
 -\nabla\cdot \left(u(x) \nabla q(x)\right) = f(x) \,\text{ in } D=(0,1)^d \;,
\quad 
    q(x) = 0 \, \text{ for } x\in\partial D\;. 
\ee
Problem~\eqref{eq:diff:diffeq} is a particular instance of the 
abstract setting introduced in Section~\ref{sec:OpEqUncInp},
with $\bcX = \bcY = H^1_0(D)$ and with deterministic right-hand side $F\in \bcY'$  
obtained by associating to $f$ a continuous, linear functional
$F$ on $\bcX = H^1_0(D)$, i.e.
\begin{equation*}
    A(u;q) = -\nabla\cdot \left(u \nabla q\right) \in \cL(H^1_0(D),H^{-1}(D)), 
\qquad F(\cdot) = \langle f, \cdot\rangle 
\;.
\end{equation*}
We now parametrize the diffusion coefficient as in \eqref{eq:uviapsi}
with the $s$-dimensional parameter vector $\bsy\in U_s=[-1,1]^s$,
which we indicate by writing $u(x,\bsy)$.
Note that the QMC integration domain $U_s = [-\frac12,\frac12]^s$
can be equivalently obtained by rescaling.
For a nominal input $\langle u \rangle \in X := L^\infty(D)$, 
a finite truncation dimension $s$
and a sequence $(\psi_j)_{j\geq 1} \subset X$, 
we consider affine-parametric uncertainties
\begin{equation}\label{eq:diff:param}
    u(\cdot,\bsy)
    =
   \langle u \rangle+ 
    \sum_{j=1}^s y_j \psi_j(\cdot)
\;.
\end{equation}
We assume that the nominal coefficient 
$\langle u \rangle$ is positive and bounded,
\be\label{eq:baruBd}
0 < \bar{u}_{\min} \leq \langle u \rangle \leq  \bar{u}_{\max} \;.
\ee
We suppress in the following the explicit dependence on $x$. 
We also assume the \emph{sparsity condition}
$(\|\psi_j\|_{L^\infty(D)})_j\in\ell^r(\bbN)$ 
for some $0<r<1$.
%
The parametric weak formulation reads:
for $\bsy\in U_s$ find $q(\cdot,\bsy)\in H^1_0(D)$ such that
for all $v \in H^1_0(D)$ holds
\begin{equation}\label{eq:diff:weak}
    \int_D \! u(x,\bsy)
        \nabla_x q(x,\bsy)
        \cdot
        \nabla_x v(x)
        \,\dd x
    =
    \int_D \! f(x) v(x) \,\dd x
\;.
\end{equation}
In \cite{CCS2}, Problem~\eqref{eq:diff:weak}, with $U$ given by \eqref{eq:diff:param},
was considered for complex parameter sequences 
$\bsz = (z_j)_{j\geq 1}$ under the following assumption.
\begin{assumption}[Uniform Ellipticity] 
\label{asmp:UE}
There exist constants $ 0 < u_- < u_+ < \infty$ 
such that for a.e. $x\in D$ and for all 
$\bsz\in \calU_s=\{ z_j\in \IC: |z_j|\leq 1, 1\le j \le s \}$, 
$s\in \IN$, there holds
\begin{equation} \label{eq:UE+-}
    0 < u_- \le \Re u(\bsz) \le u_+ < \infty \;.
\end{equation}
\end{assumption}
We remind the reader that $\langle u \rangle$ 
and $\psi_j$ are assumed to be real-valued functions.  
The choice $\bsz = \bsnull$ 
in Assumption~\ref{asmp:UE} and \eqref{eq:baruBd} 
imply $u_- \leq \langle u \rangle \leq u_+$.  
A sufficient condition for Assumption~\ref{asmp:UE} to hold is 
\eqref{eq:baruBd} and the condition
\be\label{eq:UE1}
\gamma := \left\| \frac{1}{ \langle u \rangle} 
\sum_{j\geq 1} |\psi_j| \right\|_{L^\infty(D)} < 1\;.
\ee
The Lax--Milgram lemma implies, with \eqref{eq:UE+-}, for every fixed $\bsz\in \calU_s$,
and for every $s\in \bbN$
the existence of a unique solution to the variational problem \eqref{eq:diff:weak}.
Symmetric Galerkin discretization as explained in Section \ref{sec:Discr},
yields for each parameter instance $\bsy\in U_s$ 
a unique, parametric Galerkin solution 
$q^h(\bsy) \in \bcX^h\subset \bcX = H^1_0(D)$.
The family $\{ q^h(\bsy) : \bsy \in U_s \}\subset \bcX^h$ is uniformly bounded
in $\bcX$ with respect to $s$, $h$ and $\bsy\in U_s$.
\subsection{Parametric Regularity}
\label{sec:ParReg}
The parametric 
solution family $\bsz \to q(\bsz)$ of 
the linear forward problem \eqref{eq:diff:diffeq}
with complex-parametric input \eqref{eq:diff:param}
is, for any value of $s$, 
holomorphic in the sense of Definition \ref{def:peanalytic}.
To verify this, we recall from \cite[Eqn. (2.8)]{CDS2}
the notion of \emph{$\delta$-admissibility} of a poly-radius
$\bsrho = (\rho_j)_{j\geq 1}$ with $\rho_j > 1$: 
$\bsrho$ is called $\delta$-admissible if there exists a $\delta > 0$ such that
\be\label{eq:DelAdm}
\sum_{j\geq 1} \rho_j |\psi_j(x)| \leq 
    \langle u \rangle -\delta \quad \mbox{for almost all} \quad x\in D\;.
\ee %
As above, $\delta$-admissibility \eqref{eq:DelAdm} of the poly-radius $\bsrho$ 
implies \emph{$\bsrho$-weighted ellipticity}
\be\label{eq:UErho}
\gamma(\bsrho) 
:= 
\left\| \frac{1}{\langle u \rangle} \sum_{j\geq 1} \rho_j |\psi_j| \right\|_{L^\infty(D)} 
< 1\;.
\ee
Under \eqref{eq:UE1}, the sequence $\rho_j = 1$ is $\delta$-admissible for some $\delta > 0$ 
and \eqref{eq:UErho} and \eqref{eq:UE1} coincide.

It was shown in \cite[Lemma 2.4]{CDS2} that under 
condition \eqref{eq:DelAdm} or \eqref{eq:UErho},
the family of parametric solutions 
$\{ q(\bsy) \in \bcX: \bsy \in U\}$
admits a holomorphic extension to the polydisc 
$\calU_\bsrho := \prod_{j\geq 1} \{ |z_j|<\rho_j \}\subset \IC^{\IN}$.
In particular, for every fixed $\bsy \in U$, 
$q(\bsy)$ admits an extension with respect to $\bsy_{\{1:s\}}$ 
to the finite dimensional cylinder 
$\calU_{\bsrho,s} :=  \prod_{1\leq j\leq s} \{ |z_j|<\rho_j \}\subset \IC^{s}$
for any finite dimension $s\in \IN$ with modulus $\| q(\bsz) \|_{\bcX}$ uniformly
bounded with respect to $\bsz\in \calU_{\bsrho,s}$, and uniformly with respect to 
$s\in \IN$.

In particular, for uncertainty parametrizations \eqref{eq:diff:param}
with sequences $(\psi_j)_{j\geq 1}$ satisfying \eqref{eq:UErho} with
$\bsrho\in \ell^p(\IN)$, Theorem \ref{thm:DsiboundC} then implies
that the parametric solution family $q(\bsz)$ is $(\bsb,p,\eps)$-holomorphic
with 
$\bsb = \left( \frac{1} {\langle u \rangle} \|\psi_j\|_{L^\infty(D)}  \right)_{j \ge 1}$.
This in turn implies by Proposition \ref{thm:pdTheta} that the parametric
Bayesian posterior densities $\Theta(\bsy)$ and $\Psi(\bsy)$ are 
$(\bsb,p,\eps)$ holomorphic, so that by Theorem \ref{thm:DsiboundC}
the parameter derivatives of $q(\bsy)$, $\Theta(\bsy)$ and $\Psi(\bsy)$ 
satisfy, for $\bsy\in U_s$, with any 
finite parameter dimension $s\in \IN$ the bounds \eqref{eq:HybBd} and \eqref{eq:FsPODBd}.
\subsection{Example 1: \KL series}
\label{sec:ExplKL}
Here, we choose $d=1$, $D=(0,1)$, $\langle u \rangle = 1$ 
and, for some parameter $\zeta > 1$ to be specified,
\be\label{eq:TrigBas}
\psi_{2j}(x) = (2j)^{-\zeta} \sin(j\pi x),\quad \psi_{2j-1}(x) = (2j-1)^{-\zeta} \cos(j\pi x)
\ee
Then for every finite $s$, the $s$-term expansion \eqref{eq:diff:param} 
is smooth with respect to $x$; this regularity is, however, not uniform
with respect to the truncation dimension $s$:
the choice of $\zeta$ in \eqref{eq:TrigBas} limits the 
spatial Sobolev regularity of $u(\bsy)$ 
in Sobolev and H\"older scales in the physical domain $D$.
Here we have \eqref{eq:UE1} and
$\bsb = (\| \psi_j \|_{L^\infty(D)})_{j\geq 1}\in \ell^p(\IN)$
holds for any $1\geq p > 1/\zeta$.
\subsection{Example 2: Indicator functions}
\label{sec:Indic}
When $\psi_j$ are indicator functions of sets $D_j \subset D$
which form a partition of $D$ and could
be viewed as a model for material mixtures.
We choose the basis in parametric space $\{\psi_j(x)\}_{j=1}^s$
as ``step functions'' with supports $D_j$, i.e.
\begin{equation}\label{eq:stepbasis}
    \psi_j(x) = b_j \chi_{D_j}(x),  \; b_j = \theta j^{-\zeta}, 
\; 0\leq \theta \leq u_- / 4 \;,
\end{equation}
where $\chi_I$ denotes the characteristic function of the interval $I\subset D$
and where $\{ D_j \}_{j=1}^s$ denote a partition of the physical domain $D$.
For the choice \eqref{eq:stepbasis} of $\psi_j$, 
\eqref{eq:UErho} holds with the particular sequence 
\be\label{eq:CharFctrho}
\rho_j 
= 1 + \frac{u_-}{4\|\psi_j\|_{L^\infty(D)}} 
= 1 + \frac{u_-}{4b_j}
\;. 
\ee
This follows from 
$$
\left\| \sum_{j=1}^s \rho_j \psi_j \right\|_{L^\infty(D)} 
= 
\max_{1\leq j \leq s} \left( 1 + \frac{u_-}{4b_j} \right) b_j 
=
\frac{u_-}{4} + \| \bsb \|_{\ell^\infty} 
\leq 
\frac{u_-}{2} 
\;.
$$
In this case 
$E = \bbN$ and 
$E^c=\emptyset$  in Proposition~\ref{prop:main1}
so that HoQMC weights in \eqref{equ:hybridWeight} become product weights.
This, in turn, implies in Proposition~\ref{prop:main1} that 
the complexity of the CBC construction for the HoQMC generating vectors
scales linearly with the parameter dimension $s$. 
\section{Numerical Results}\label{sec:results}
\subsection{Construction of Interlaced Polynomial Lattice Rules}
\label{sec:CBC}
We consider the construction of interlaced polynomial 
lattice rules introduced in \cite{DKGNS13}
by the component-by-component (CBC) algorithm.
Proposition~\ref{prop:main1} shows that the CBC construction is feasible in
$\mathcal{O}(\alpha^2s^2N\log N)$ operations for SPOD weights
and in $\mathcal{O}(\alpha sN\log N)$ operations for product weights,
which were both verified computationally in \cite{GaCS14}.
We consider the examples mentioned in Sections~\ref{sec:ExplKL}~and~\ref{sec:Indic} above,
which are of the SPOD and product type, respectively.
The SPOD weights case corresponds to an empty set $E$ in Proposition~\ref{prop:main1};
for the product weights case we have $J=\infty$, i.e.~$E^c$ is empty.

In a computer implementation of the CBC algorithm for 
polynomial lattice rules, it is most natural to use the base $b=2$,
since then polynomials over $\mathbb{Z}_2$ of degree less than $m$ can be represented
as bit sequences of length $m$, allowing for very efficient bit-wise manipulations.
In \cite{Yoshiki15} it was shown that for $b=2$, we have $C_{\alpha,b}=1$,
which yields significantly improved generating vectors
(in terms of the observed convergence of the integration error for suitable test integrands)
compared to the previous bounds with $C_{\alpha,b}>1$.
Further improvement is observed in practice by 
reducing the numerical value of $C$ used in the CBC construction;
in all experiments below, we use generating vectors based on the choice $C=0.1$.
A computational study of the impact on the choice of the Walsh
constant was performed in \cite{GaCS14}.
\subsection{Approximation of Prior Expectation}
\label{sec:AppPrior}
We consider here forward uncertainty quantification under the
assumption of uniformly distributed parameters $\bsy\in U_s$,
i.e.~we choose the prior distribution $\bpi_0 = \lambda^s$,
where $\lambda^s$ denotes the $s$-dimensional Lebesgue measure.
We denote in the following by $q^h(x,\bsy)$ the finite element approximation
of the solution $q$ of \eqref{eq:diff:diffeq} with discretization parameter $h$.
The goal of the computation is then the approximation of the expectation
of a quantity of interest (QoI) function $\phi(\bsy)$.
We choose here and in the following as the quantity of interest
the point evaluation of $q^h$ at the point $\bar{x}=0.25$,
$\phi(\bsy) = q^h(\bar{x},\bsy)$,
and thus seek to compute the approximation
\begin{equation}
    E^{\bpi_0}[\phi] :=
    \frac{1}{N} \sum_{n=0}^{N-1} q^h(\bar{x};\bsy^{(n)})
    \quad\approx\quad
    \mathbb{E}^{\bpi_0}[\phi] =
    \int_{U_s} \! q^h(\bar{x};\bsy) \,\dd\bsy
    ,
\end{equation}
where $\{\bsy^{(n)}\}_{n=0}^{N-1}$ is the interlaced polynomial lattice point set.
We are interested in the rate of convergence of the quadrature approximation $E^{\bpi_0}$
to the true value $\mathbb{E}^{\bpi_0}$.
Additionally, we consider the convergence of the
dimension truncation and finite element errors, where applicable.

\subsection{Approximation of Posterior Expectation}
The approximation of the Bayesian inverse, as introduced in Section \ref{sec:BayInv},
is computationally very similar to the approximation of the prior expectation,
since we must compute the ratio estimate $\mathbb{E}^{\bpi^\delta}[\phi]=Z'/Z$,
where $Z$ and $Z'$ are high-dimensional integrals given by \eqref{eq:Z} and \eqref{eq:intpsi}.
We apply the same higher-order QMC rule mentioned above to both of these integrals.

Note that when using the same QMC rule for both integrals,
the forward model is evaluated at exactly the same parameters
$\bsy^{(0)},\ldots,\bsy^{(N-1)}$ in both integrands;
this allows a simple optimization leading to a reduction of the computational work by a factor two.
Additionally storing the result of the prior expectation in the iteration over the
quadrature points allows both forward and inverse UQ problems
to be solved in one run,
still requiring only $N$ total evaluations of the forward model.

We briefly describe our choice of observation operator and specific noise model.
As observation operator, we consider the evaluation of the solution $q(\cdot,\bsy)$
at the $K=3$ points $\bsx_{obs}=(0.2,0.5,0.7)$.
As in the exposition in Section~\ref{sec:BayInv} above,
we consider measurements perturbed by additive Gaussian noise $\eta\sim\mathcal{N}(0,\Gamma)$
with known covariance matrix $\Gamma\in\IR^{K\times K}_{sym}$.
In the experiments,
we consider one fixed instance of the measurement,
i.e.~we generate one realization of $\eta\in\IR^3$ which determines
the measurement used throughout the computation.
The measurement $\delta\in\IR^3$ is then given for the unknown,
exact value of the parameter $\bsy^\star$ as
\begin{equation*}
    \delta =
    \begin{pmatrix}
        q(0.2,\bsy^\star)+\eta_1 \\
        q(0.5,\bsy^\star)+\eta_2 \\
        q(0.7,\bsy^\star)+\eta_3
    \end{pmatrix}
    .
\end{equation*}
We assume independent noise for each component and unit variance,
giving as covariance matrix simply the identity $\Gamma=\mathbb{I}_{3\times3}$.
The parameters of the diffusion equation were chosen such that this makes sense,
i.e.~the size of the range of values of the observations are roughly $\calO(1)$.
As output quantity of interest,
the solution evaluated at the point $\bar{x}=0.25$ is used,
as in the prior expectation above.

\subsection{Results for Example 1: \KL Series}
\label{sec:res:KL}
We consider the trigonometric basis \eqref{eq:TrigBas}
where $\bar{u}$ in \eqref{eq:diff:param} is chosen to be a constant
of the size necessary to fulfill the uniform ellipticity conditions \eqref{eq:UE+-}.
For the \KL basis functions \eqref{eq:TrigBas},
we have $\|\psi_j\|_{L^\infty(D)} = j^{-\zeta}$, $\rho_j=1$,
and thus $(\|\psi_j\|_{L^\infty(D)})_j\in\ell^p(\bbN)$ 
with $ 1 \geq p > \frac{1}{\zeta}$.
The CBC construction of the generating vectors was based on 
SPOD weights (i.e., \eqref{equ:hybridWeight} with $E=\emptyset$), given by
\begin{equation*}
    \gamma_\setu :=  \sum_{\bsnu_\setu \in \{1:\alpha\}^{|\setu|}}
    |\bsnu_{\setu}|!\,
    \prod_{j\in\setu} \left(2^{\delta(\nu_j,\alpha)}\beta_j^{\nu_j}\right)
    ,
\end{equation*}
where $\delta(\nu_j, \alpha) = 1$ if $\nu_j = \alpha$ and $\delta(\nu_j, \alpha) = 0$ otherwise 
and with the sequence $\beta_j=\theta j^{-\zeta}$ with $\theta=0.2$ and $\zeta=2$.
We thus expect the QMC quadrature approximation to both prior and posterior expectation
to converge with $s$-independent rate $\calO(N^{-\zeta})$ for 
digit interlacing parameter $\alpha \geq 2$ in \eqref{alpha}.
To compare these QMC based results to the performance of existing methods,
we also apply standard Monte Carlo sampling consisting of realizations of
uniformly distributed points in $U_s$; as a (rough) work measure we 
use in either case the number $N$ of samples, which coincides with 
the number of forward PDE solves. As the algorithms considered in the 
present paper are single level algorithms, all PDE solves were performed
with equal accuracy; multi-level extensions of either method
are available and are known to deliver improved work versus accuracy
\cite{KSS13,DKGS14}.

We briefly list the parameters used in the simulations.
As right-hand side forcing term, we consider the function $f(x)=100x$.
Unless otherwise specified, we use a solution with $N=2^{20}$ quadrature points
and meshwidth $h=2^{-20}$ as a reference.
In the Monte Carlo results,
we approximate the $L^2$ error using $10$ repetitions of the estimator.
In the Bayesian inverse problem,
we assume the measurement errors to result from a normal distribution with unit variance.


\subsubsection{Finite Element Approximation}
We solve the PDE \eqref{eq:diff:diffeq} by the finite element method
with piecewise linear basis functions
and meshes obtained by regularly refining an initial mesh
consisting of the points $\{0,1\}$.
The PG discretization error of the posterior approximation,
$|\mathbb{E}^{\mu^\delta}[q(\bar{x};\bsy)] - \mathbb{E}^{\mu^\delta}[q^h(\bar{x};\bsy)]|$,
is measured by replacing $q(\bar{x};\bsy)$ with a reference solution
obtained on a mesh with meshwidth $h=2^{-20}$.
Since as QoI we consider only the evaluation of the solution at the point $\bar{x}$,
we only consider the absolute value of the (scalar) results.
For QMC quadrature, we use $N=2^{10}$ points.
The convergence results of the finite element error of the posterior approximation,
as well as of the individual integrands $Z$ and $Z'$
are shown in Figure \ref{fig:fem}.

\subsubsection{Dimension Truncation}
\label{sec:DimTrc}
To numerically verify the convergence rate of the error committed by 
dimension truncation to a finite dimension $s<\infty$,
we consider the QMC-PG approximation of $\mathbb{E}_N^{\bpi^\delta}[\phi(q^h_s)]$ for varying $s$.
In order to be able to neglect the other two error contributions,
the finite element meshwidth is chosen as $h=2^{-20}$ and $N=2^{20}$ QMC points are used.
By \eqref{eq:DTbound}, we expect a convergence rate of $s^{-(1/p-1)}$;
for the case $\zeta=2$ in \eqref{eq:stepbasis}, 
we have $p>1/2$, giving an expected rate of $s^{-1+\varepsilon}$ for an $\varepsilon>0$.
In Figure~\ref{fig:strunc}, this expected convergence rate can be clearly seen.

\begin{figure}[H]
    \centering
    \begin{minipage}[b]{0.49\linewidth}
        \centering
        \includegraphics[width=\textwidth]{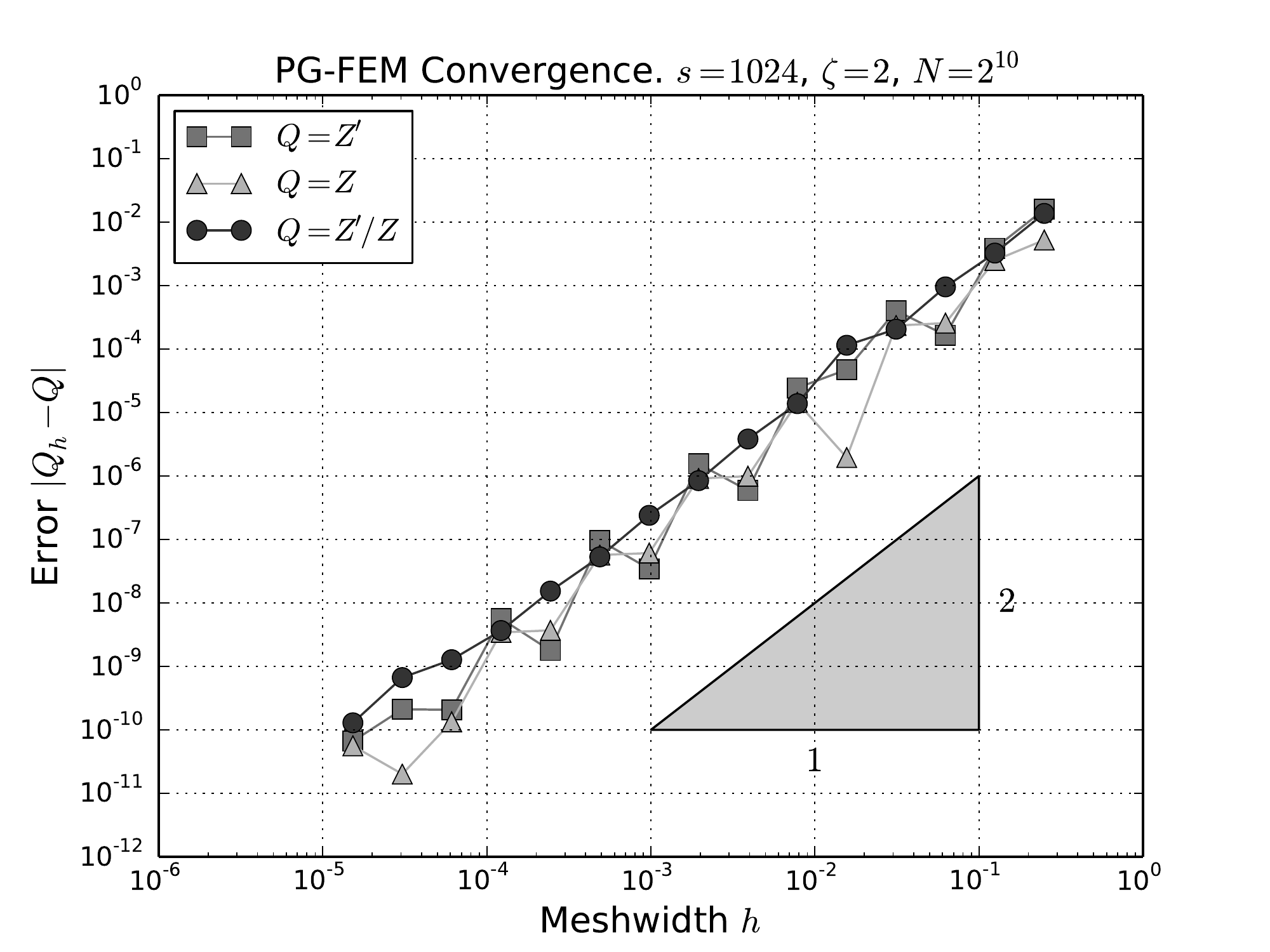}
        \subcaption{FEM error, KL basis}
        \label{fig:fem}
    \end{minipage}
    \begin{minipage}[b]{0.49\linewidth}
        \centering
        \includegraphics[width=\textwidth]{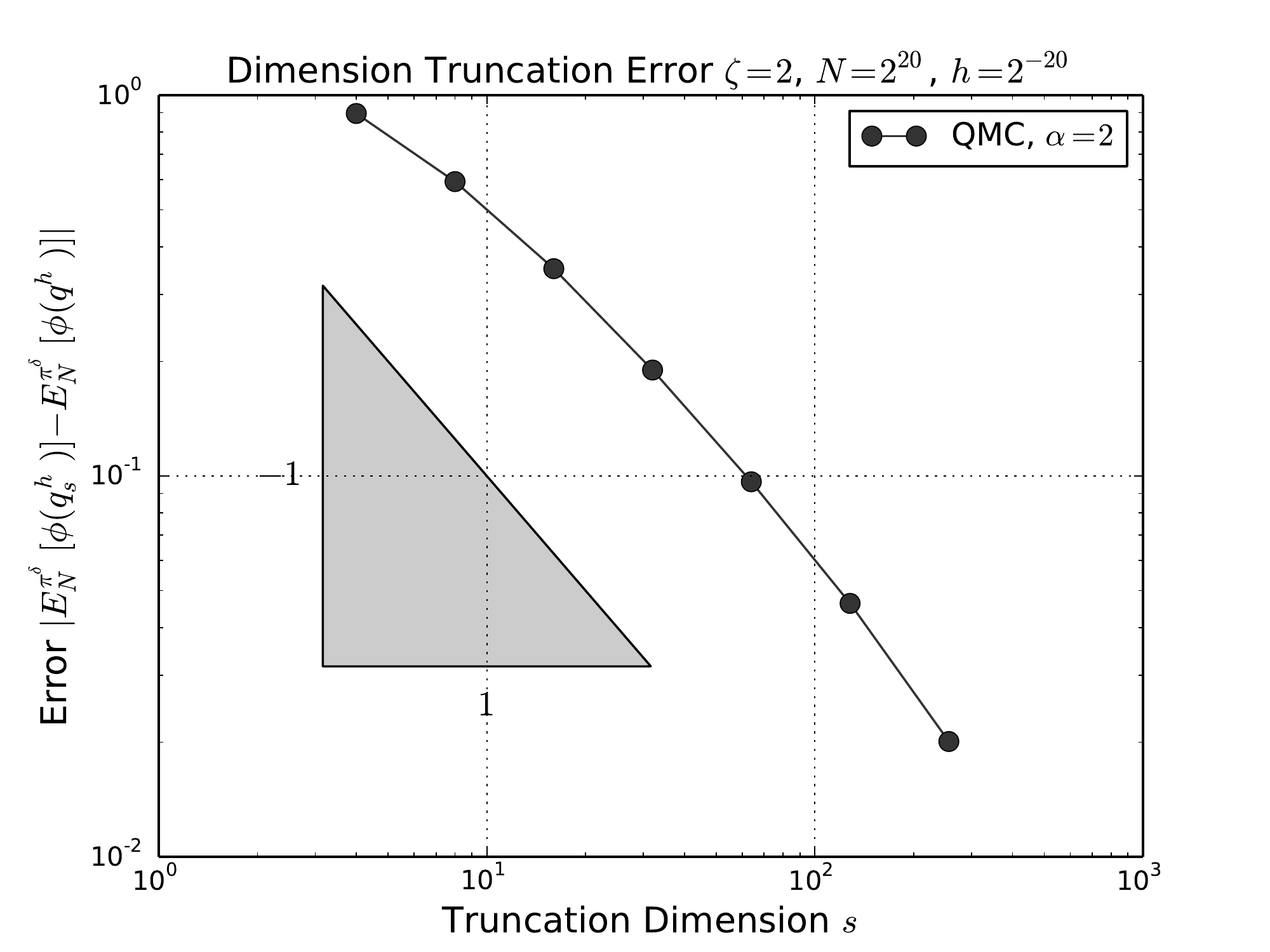}
        \subcaption{Dimension truncation error}
        \label{fig:strunc}
    \end{minipage}
    \caption{FEM and dimension truncation errors.
        Reference solutions were obtained using $h=2^{-20}$ for (a) and $s=1024$ dimensions for (b).
    }
    \label{fig:struncfem}
\end{figure}

\subsubsection{QMC Convergence}
\label{sec:QMCCnv}
Figures~\ref{fig:conv:KLforward} and \ref{fig:conv:KLinverse}
show the convergence of the QMC approximation
to the prior and posterior expectations, respectively.
In both cases, the convergence rate $N^{-2}$ is clearly visible
for the considered interlaced polynomial lattice rule
with interlacing factors $\alpha=2,3$.
This rate of convergence is in particular independent of the
dimension $s$ of the parametric space $U_s$,
for the values considered here.
%

\begin{figure}[H]
    \centering
    \begin{minipage}[b]{0.49\linewidth}
        \centering
        \includegraphics[width=\textwidth]{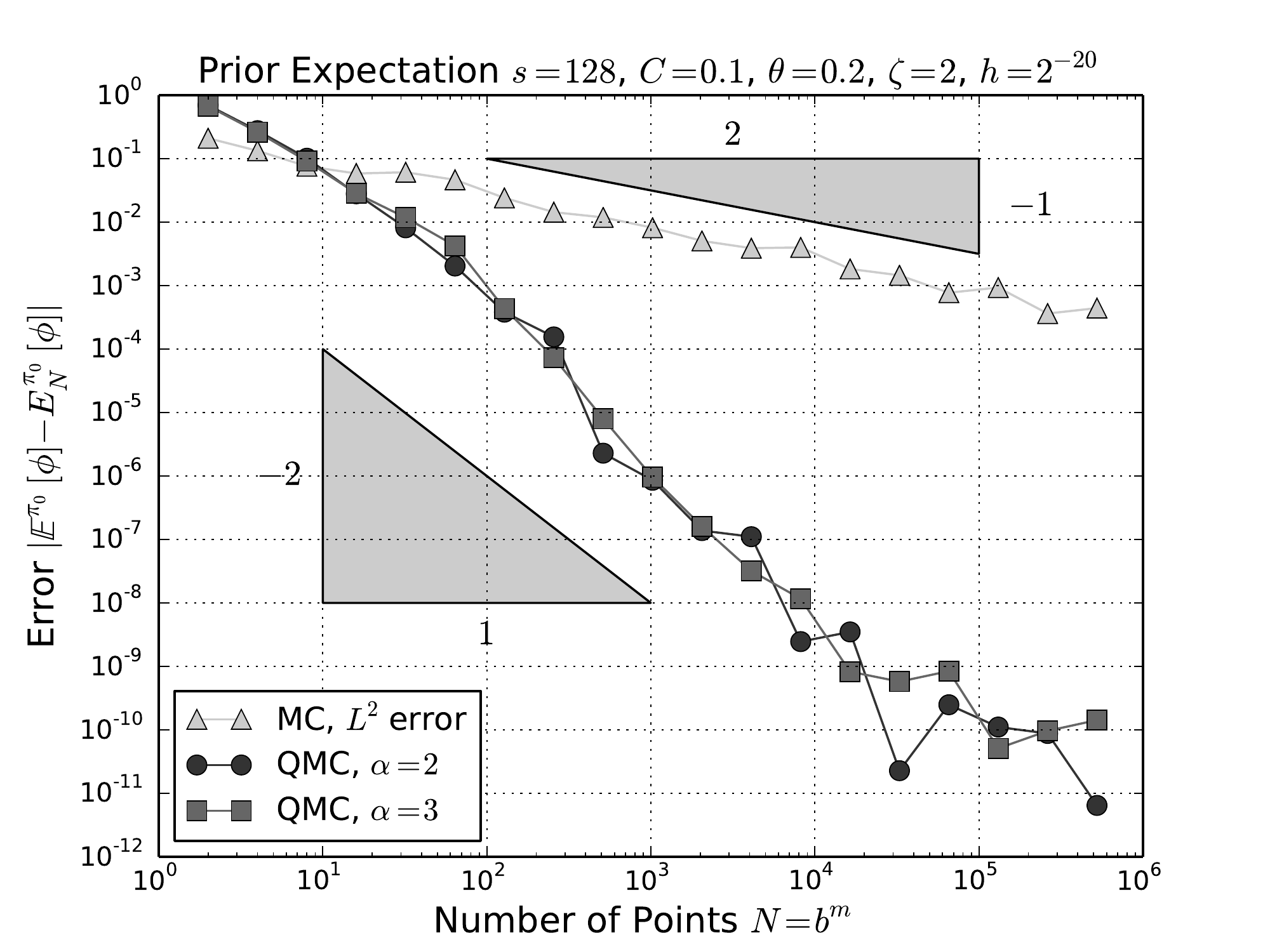}
        \subcaption{$s=128$}
    \end{minipage}
    \begin{minipage}[b]{0.49\linewidth}
        \centering
        \includegraphics[width=\textwidth]{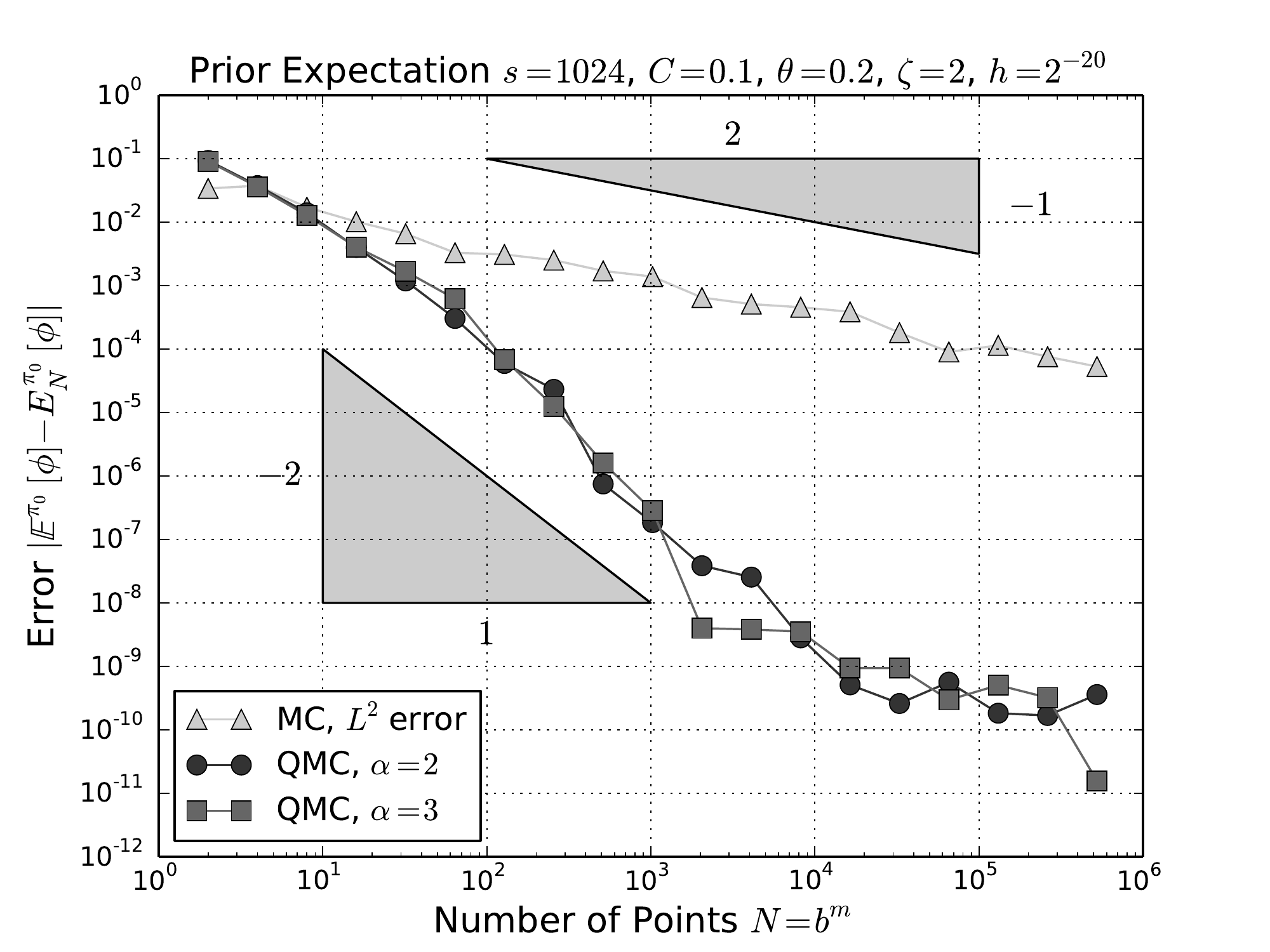}
        \subcaption{$s=1024$}
    \end{minipage}
    \caption{Convergence of the prior approximation vs.~the number of samples $N=b^m$
        for $C=0.1$, $\theta=0.2$, $s=128, 1024$, 
        $\zeta=2$, $\alpha=2,3$, $h=2^{-20}$.
        For MC, the $L^2$ error was approximated using $10$ repetitions.
    }
    \label{fig:conv:KLforward}
\end{figure}
\begin{figure}[H]
    \centering
    \begin{minipage}[b]{0.49\linewidth}
        \centering
        \includegraphics[width=\textwidth]{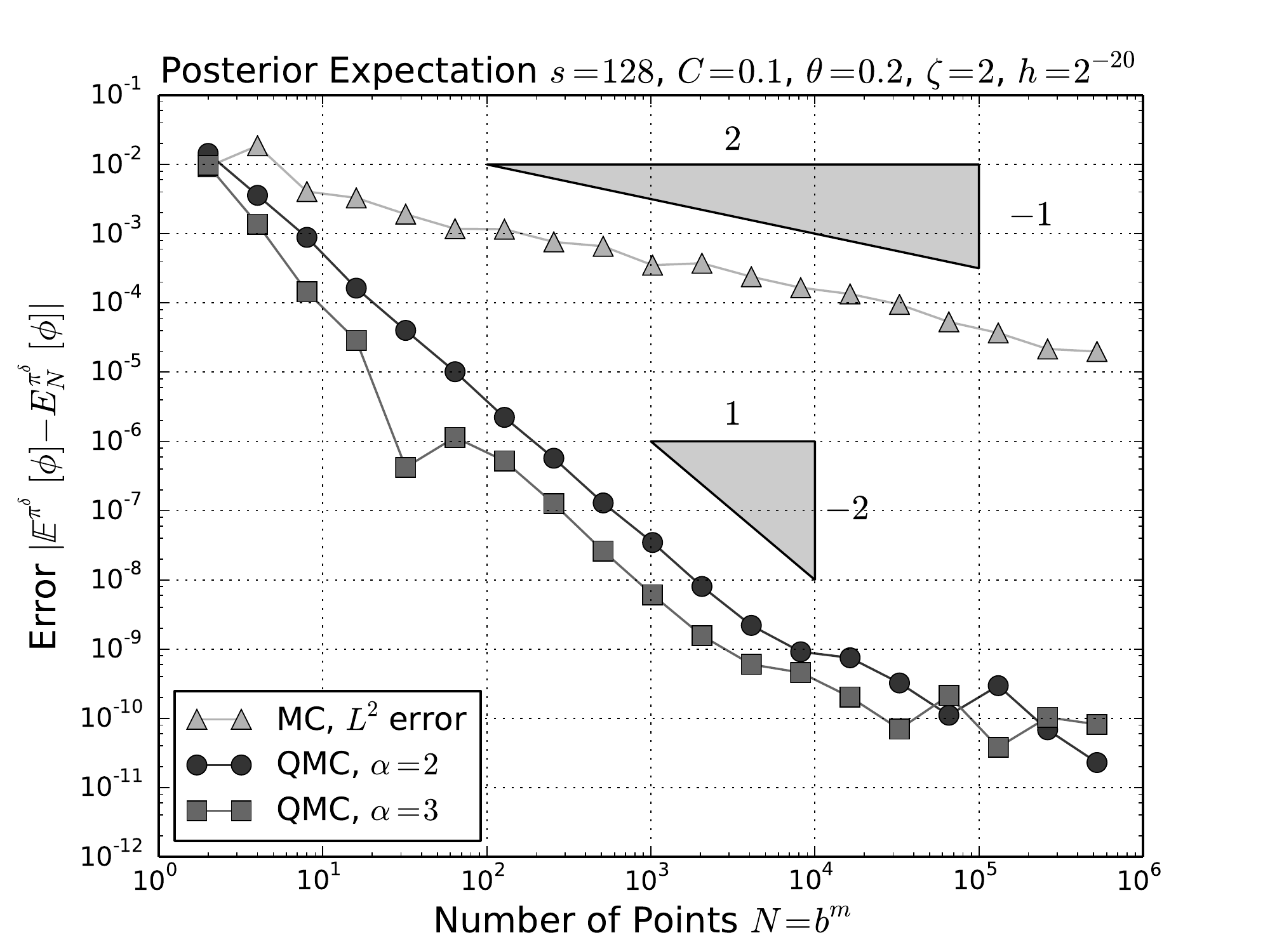}
        \subcaption{$s=128$}
    \end{minipage}
    \begin{minipage}[b]{0.49\linewidth}
        \centering
        \includegraphics[width=\textwidth]{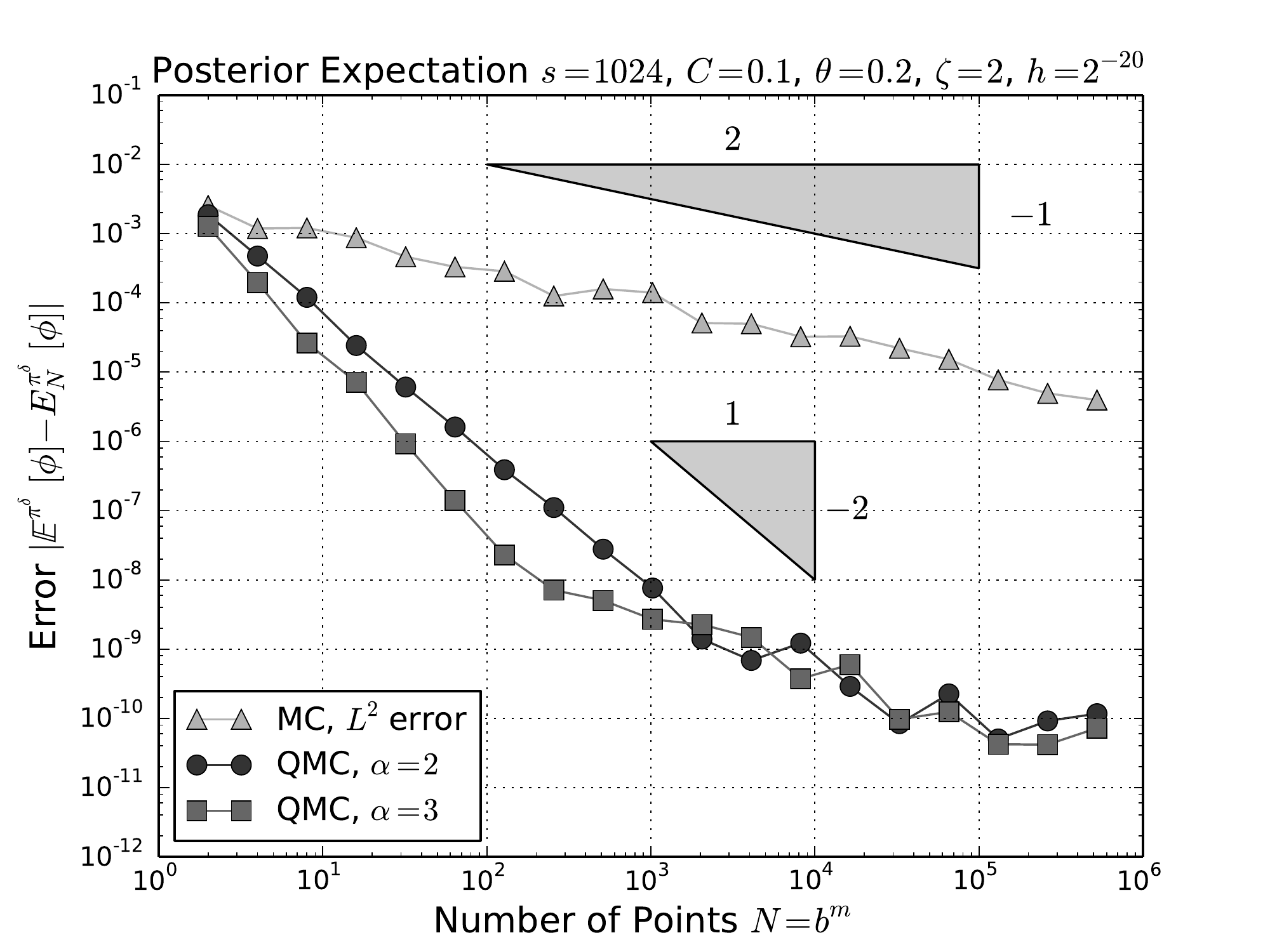}
        \subcaption{$s=1024$}
    \end{minipage}
    \caption{Convergence of the posterior approximation vs.~the number of samples $N=b^m$
        for $C=0.1$, $\theta=0.2$, $s=128, 1024$,
        $\zeta=2$, $\alpha=2,3$, $h=2^{-20}$.
        For MC, the $L^2$ error was approximated using $10$ repetitions.
    }
    \label{fig:conv:KLinverse}
\end{figure}

\subsection{Results for Example 2: Indicator Functions}
\label{sec:res:ind}

We consider the indicator function basis as in \eqref{eq:stepbasis},
with the intervals $D_j=(x_{j-1},x_j)$ chosen based on the points $x_0,\ldots,x_s$
of a graded mesh $\mathcal{T}\subset(0,1)$.
The points in $\mathcal{T}$ are obtained by transforming
an equidistant mesh with the function $g(x)=x^a$ for an $a\in\bbR$.
In the following, we use $a=0.2$, yielding the points $x_j=(j/s)^{1/5}$ for $j=0,\ldots,s$.
This choice implies that the support of the first few parametric basis functions is
relatively large, ensuring that the range of the observations of the solution is
of the same order of magnitude for the values of $s$ used in the experiments.
This, together with the choice of $f$ mentioned below,
justifies the use of $\Gamma=\mathbb{I}_{3\times3}$ in the
measurement model of the Bayesian inverse problem.
Note that for different ``truncation'' dimensions $s$ 
the choice \eqref{eq:stepbasis} with intervals $D_j=(x_{j-1},x_j)$ and  $x_j=(j/s)^{1/5}$
implies that we solve different problems; 
convergence for $s\to\infty$ is moot for this data.

In the following, we consider the right-hand side function $f(x)$ to be constant.
Together with the piecewise constant diffusion coefficient model,
this implies that the solution is a piecewise quadratic function
on the given mesh $\mathcal{T}$.
Thus, if we use quadratic element basis functions in the finite element computations,
we will obtain the exact solution, allowing us to ignore effects of the discretization error.
\subsubsection{Choice of QMC weights}
\label{sec:ChcWgt}
As mentioned in Section~\ref{sec:Indic},
this choice of diffusion coefficient model allows the use of product weights
in the norm in Definition~\ref{def_F_norm} of the form (cp.~\eqref{equ:hybridWeight} with $E^c=\emptyset$)
\begin{equation*}
    \gamma_\setu :=  \sum_{\bsnu_\setu \in \{1:\alpha\}^{|\setu|}}
    \bsnu_{\setu}!\,
    \prod_{j\in\setu} \left(2^{\delta(\nu_j,\alpha)}\beta_j^{\nu_j}\right)
    ,
\end{equation*}
which we construct with the sequence $\beta_j=\theta j^{-\zeta}$ with $\theta=0.25$ and $\zeta=2$.
This is an advantage because the number of operations required for the construction
of the generating vector is linear in the dimension $s$,
whereas for the SPOD weights used in Section~\ref{sec:res:KL} it 
scaled quadratically with respect to $s$ (see Proposition \ref{prop:main1}).
This renders problems with large parametric dimension $s$ computationally
accessible.
\subsubsection{QMC Convergence}
\label{sec:QMCConv}
Figures~\ref{fig:P2forward} and \ref{fig:P2inverse}
show the convergence of the QMC approximation
to the prior and posterior expectations, respectively.
In both cases, we observe the expected rate $N^{-2}$,
which seems to be independent of the parameter space dimension $s$
for the values of $s$ considered here.
\begin{figure}[H]
    \centering
    \begin{minipage}[b]{0.49\linewidth}
        \centering
        \includegraphics[width=\textwidth]{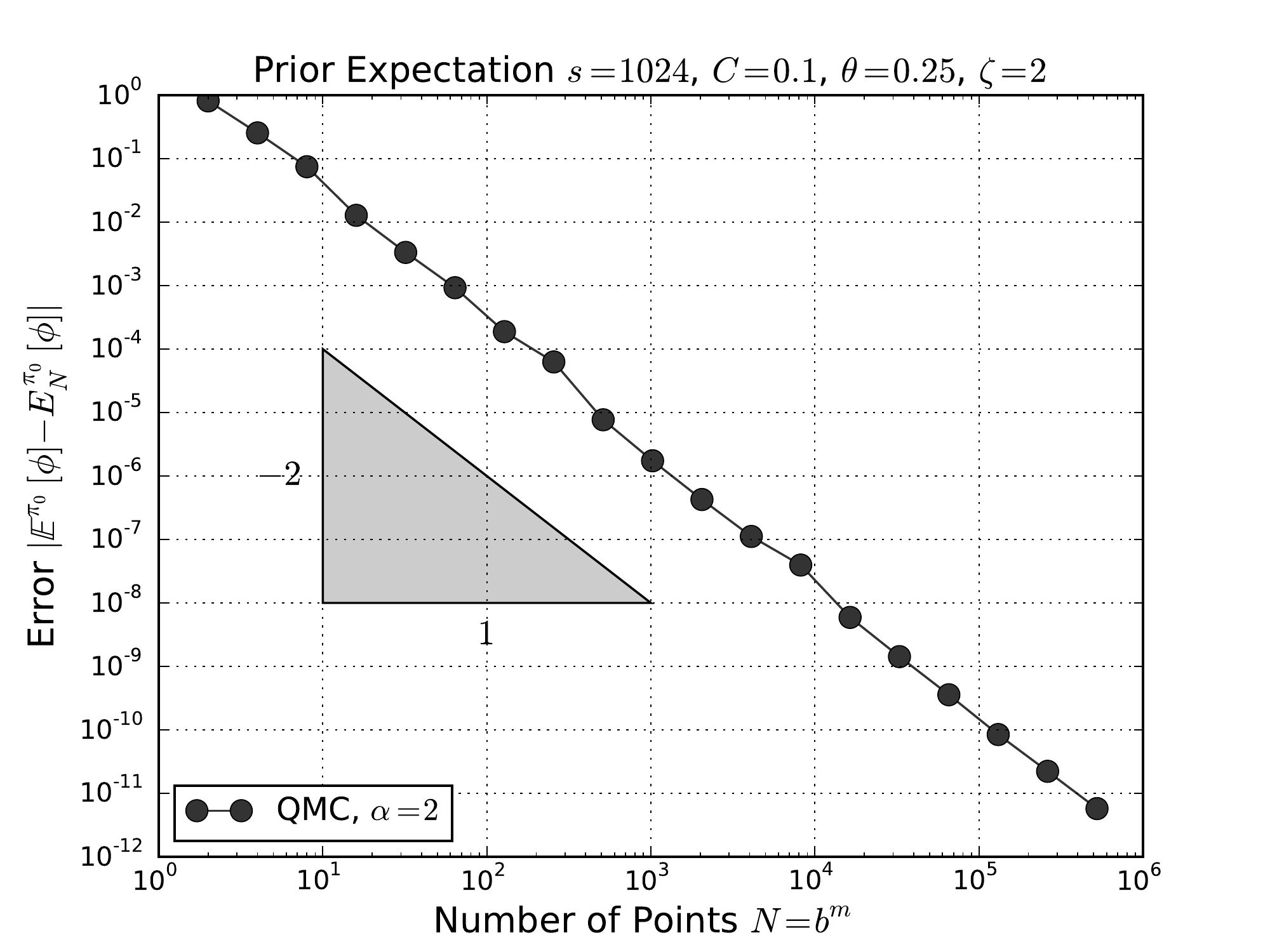} 
        \subcaption{$s=1024$}
    \end{minipage}
    \begin{minipage}[b]{0.49\linewidth}
        \centering
        \includegraphics[width=\textwidth]{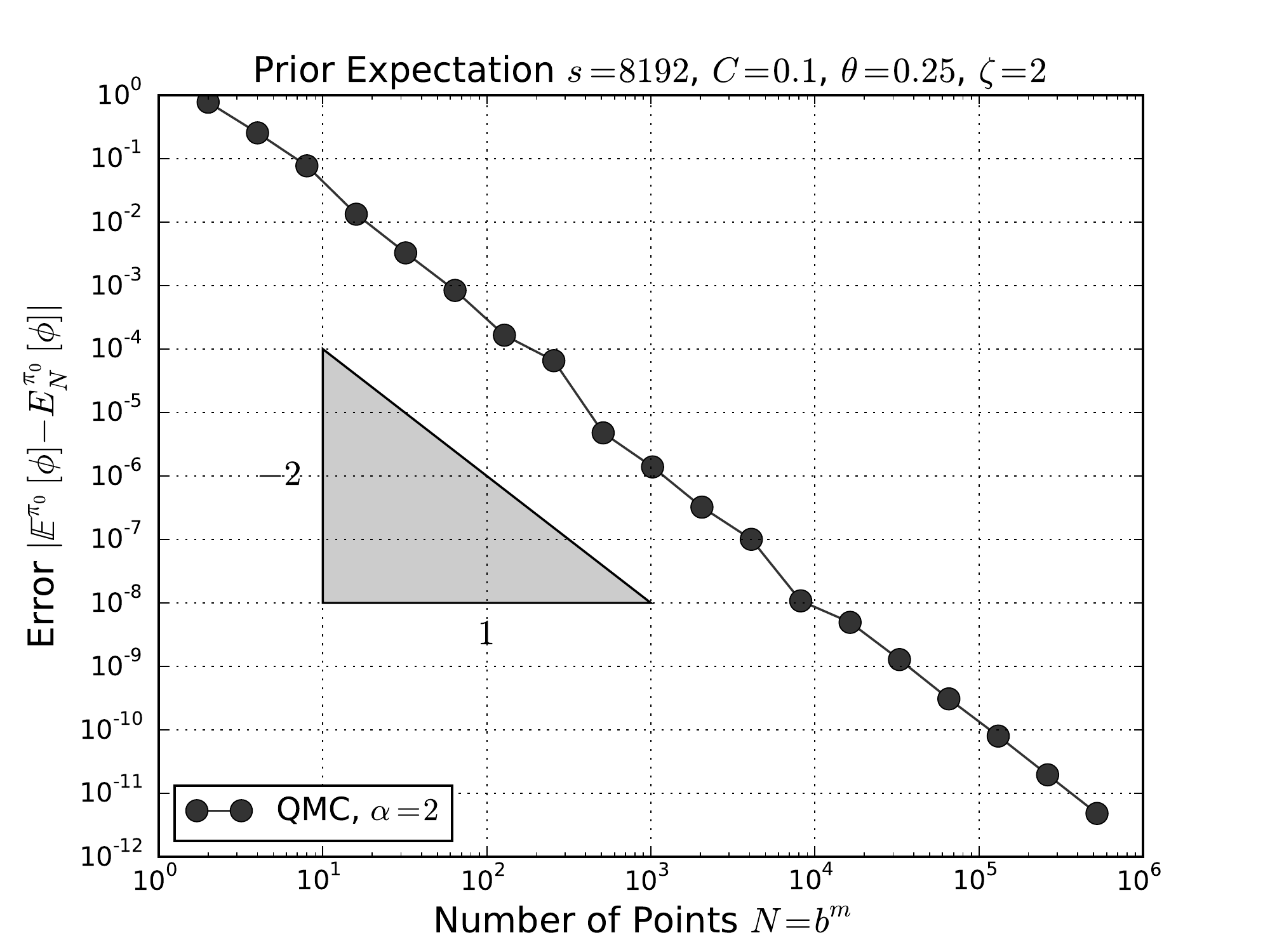} 
        \subcaption{$s=8192$}
    \end{minipage}
    \caption{Convergence of the prior approximation vs.~the number of samples $N=b^m$
        for $C=0.1$, $\theta=0.25$, $s=1024, 8192$,
        $\zeta=2$, $\alpha=2,3$.
    }
    \label{fig:P2forward}
\end{figure}
\begin{figure}[H]
    \centering
    \begin{minipage}[b]{0.49\linewidth}
        \centering
        \includegraphics[width=\textwidth]{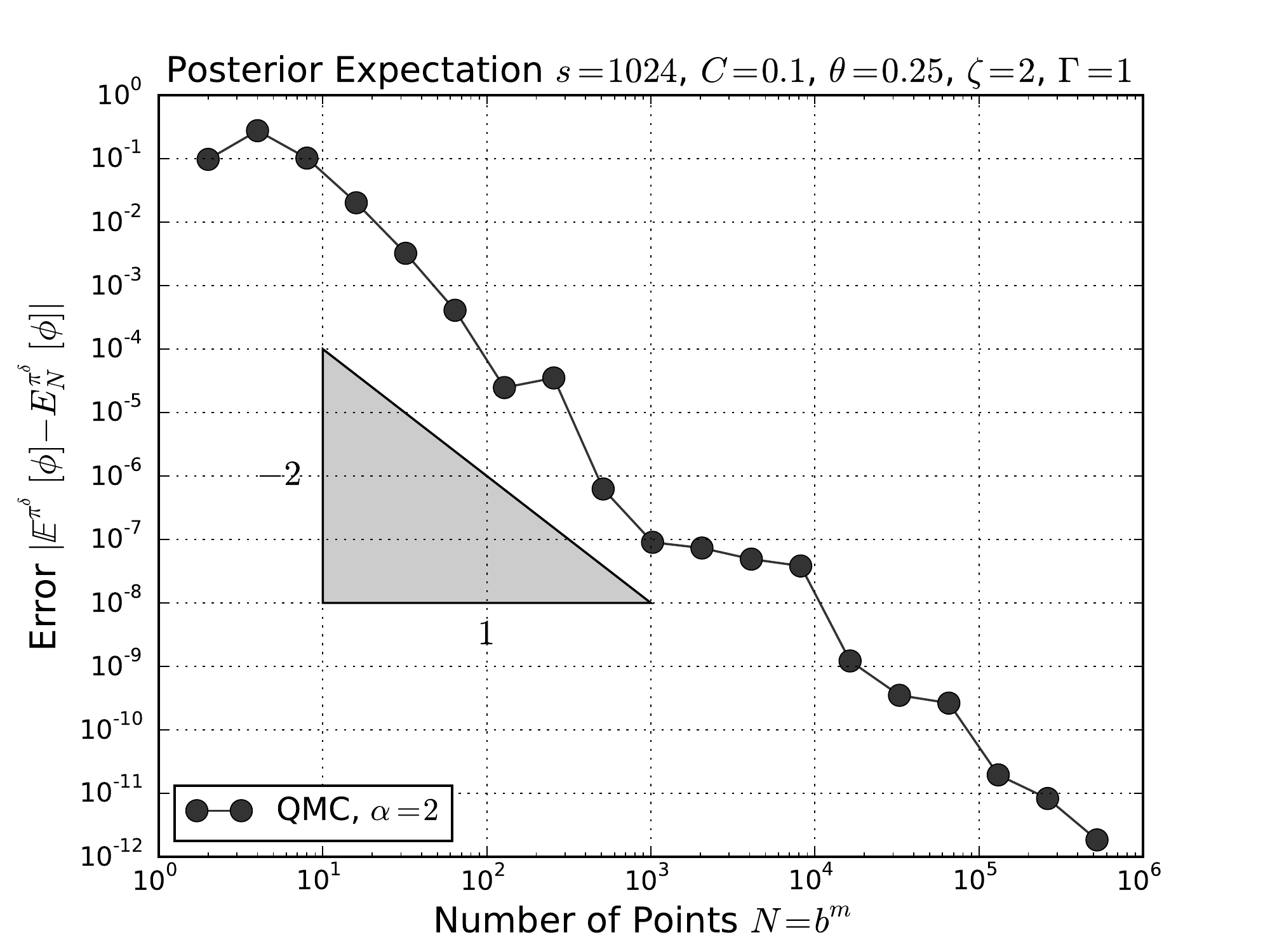} 
        \subcaption{$s=1024$}
    \end{minipage}
    \begin{minipage}[b]{0.49\linewidth}
        \centering
        \includegraphics[width=\textwidth]{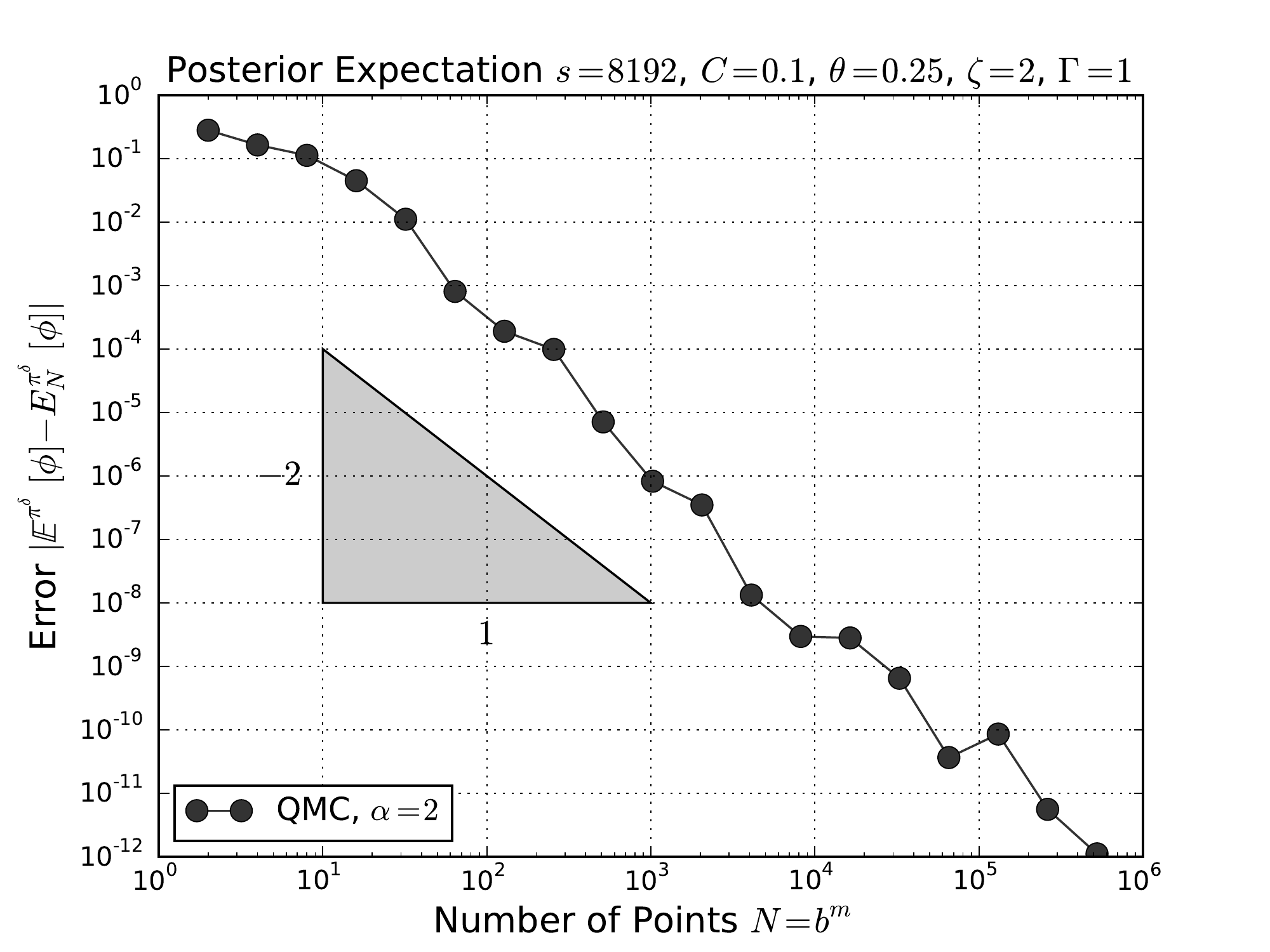} 
        \subcaption{$s=8192$}
    \end{minipage}
    \caption{Convergence of the posterior approximation vs.~the number of samples $N=b^m$
        for $C=0.1$, $\theta=0.25$, $s=1024, 8192$,
        $\zeta=2$, $\alpha=2,3$.
    }
    \label{fig:P2inverse}
\end{figure}

Comparing these convergence results with the corresponding
results for the \KL basis, the `levelling' of the total errors
at around $10^{-10}$ in Fig. \ref{fig:conv:KLforward} is due
to the spacial discretization error, which is absent in 
the presently considered indicator function basis with $P2$-Finite Elements,
suggesting that the additive structure of the combined error bound
\eqref{eq:RatEstErrBd} (which resulted from the triangle inequality)
is sharp in these cases.

\paragraph{Acknowledgments}
This work was supported by CPU time from the 
Swiss National Supercomputing Centre (CSCS) 
under project IDs s522 and d41,
by the Swiss National Science Foundation (SNF) under Grant No. SNF149819,
and by Australian Research Council's Discovery Projects under project
number DP150101770.
This work has benefited from discussions between the authors at
the conference 
``Approximation of High-Dimensional Numerical Problems: 
Algorithms, Analysis and Applications''
at BIRS, Banff, Canada, September 27 - October 2, 2015.
 

\end{document}